\documentclass[preprint,12pt]{article}
\usepackage{amsmath}
\usepackage{amsfonts}
\usepackage{amssymb}
\usepackage{amsthm}
\usepackage{amstext}
\usepackage{graphicx}
\usepackage[all]{xy}
 
\newtheorem{teorema}{Theorem}[section]
\newtheorem{lema}[teorema]{Lemma}
\newtheorem{corolario}[teorema]{Corollary}
\newtheorem{proposicion}[teorema]{Proposition}
\theoremstyle{definition}
\newtheorem{definicion}[teorema]{Definition}
\newtheorem{remark}[teorema]{Remark}

\def\div{\operatorname{div}}
\def\grad{\operatorname{grad}}
\def\curl{\operatorname{curl}}
\def\Dr{D_{\rm r}}
\def\G{\operatorname{\mathcal{G}}}
\def\H{\mathbb{H}}
\def\Hi{\operatorname{\mathcal{H}}}
\def\Har{\operatorname{Har}}
\def\Im{\operatorname{Im}}
\def\Ker{\operatorname{Ker}}

\def\C{\mathbb{C}}
\def\M{\mathfrak{M}}
\def\PV{\mbox{\;P.V.\!\!}}
\def\R{\mathbb{R}}
\def\SI{\operatorname{SI}}

\def\Sc{\operatorname{Sc}}
\def\Vec{\operatorname{Vec}}
\def\iff{\Leftrightarrow}
\def\Sol{\operatorname{Sol}}
\def\Irr{\operatorname{Irr}}
\def\Fi{F_{0,\partial\Omega}}
\def\Fii{{\overrightarrow F}_{\!\!1,\partial\Omega}}
\def\Ti{T_{0,\Omega}}
\def\Tii{{\overrightarrow T}_{\!\!1,\Omega}}
\def\Tiii{\overrightarrow T_{\!\!2,\Omega}}
\def\tr{\operatorname{\rm tr}}
\def\vecK{\overrightarrow{K}}
\def\nrm{_{\mathbf n}}
\def\tng{_{\mathbf t}}

\begin{document}
 
\begin{center}
  
  {\Large Hilbert transform for the\\ three-dimensional Vekua equation}

\vspace{2ex}

  Briceyda B. Delgado
 \\R. Michael Porter 

\vspace{2ex}

 Department of Mathematics, Cinvestav-Quer\'etaro, \\
Libramiento Norponiente 2000, Fracc.\ Real de Juriquilla,\\
 Santiago de Quer\'etaro, C.P.\ 76230 M\'exico

\end{center}





 
\begin{abstract}
  The three-dimensional Hilbert transform takes scalar data on the
  boundary of a domain $\Omega\subseteq\R^3$ and produces the boundary
  value of the vector part of a quaternionic monogenic
  (hyperholomorphic) function of three real variables, for which the
  scalar part coincides with the original data. This is analogous to
  the question of the boundary correspondence of harmonic
  conjugates. Generalizing a representation of the Hilbert transform
  $\Hi$ in $\R^3$ given by T.\ Qian and Y.\ Yang (valid in $\R^n$), we
  define the Hilbert transform $\Hi _f$ associated to the main Vekua
  equation $DW=(Df/f)\overline{W}$ in bounded Lipschitz domains in
  $\R^3$. This leads to an investigation of the three-dimensional
  analogue of the Dirichlet-to-Neumann map for the conductivity
  equation.
\end{abstract}

  
\noindent Classification:  44A15 30E20 30G20 35J25 35Q60 

\section{Introduction}
The aim of this paper is to show the existence of a natural ``Hilbert
transform'' $\Hi_f$ associated to the main Vekua equation
\eqref{eq:Vekua} in a bounded Lipschitz domain $\Omega\subseteq\R^3$,
in the generality of solutions in the Sobolev space
$H^{1/2}(\partial\Omega)$. This is a system of real equations in the
dress of a quaternionic formula.  The scalar part of a solution of the
main Vekua equation satisfies a conductivity equation, while the
vector part satisfies a double curl-type equation coupled with the
condition of being divergence free (see \eqref{eq:conductivity} and
\eqref{eq:doublerot} below).  Our construction of $\Hi_f$ is inspired
by the Hilbert transform given by T.\ Qian and others for the
monogenic case of the Vekua equation, defined in terms of the
component operators of the principal value singular Cauchy integral
operator and an inverse operator related to layer potentials
\cite{AKQ2009,Qian2008,Qian2009}. In the literature the Hilbert
transform has sometimes been mistakenly identified with the vector
part of the boundary value of the Cauchy integral, since they happen
to coincide for half spaces in $\R^n$ \cite[p.\ 758]{Qian2009} and for
the unit disk in the plane \cite[Example 2.7(2)]{AKQ2009}. However,
this does not hold for general domains, including higher dimensional
balls \cite[Example 2.7(3)]{AKQ2009}.
 
The conductivity equation describes the behavior of an electric
potential in a conductive medium. In 1980, A.\ P.\ Calder\'on
\cite{Calderon1980} posed the question of whether it is possible to
determine the electrical conductivity of a medium by making
measurements at the boundary. Results obtained since then on the
solvability, stability, uniqueness, and other properties of the
Dirichlet problem associated to this kind of elliptic second order
differential equation in $\R^n$ for $n\geq 3$ (e.g.\ Lemma
\ref{lema:solution_conductivity} below and
\cite{Kenig2007,Uhlmann2009}) will be essential in the development of
the present work. This inverse problem is the subject of Electrical
Impedance Tomography; for more about medical applications of the
conductivity equation see \cite{Holder2005}.  A series of articles of
Brackx et al.\ \cite{Brac2006-1,Brac2006-2, Brac2008} study the
relationship between the Hilbert transforms and conjugate harmonic
functions in the context of Clifford algebras on the unit sphere.

The outline of the paper is as follows. In Section
\ref{sec:preliminaries} we present the notation with basic facts of
quaternionic analysis. In Section \ref{sec:Hilbert-monogenic} some
operator properties related to boundedness and invertibility of the
Hilbert transform $\Hi$ are given, as well as an explicit form for its
adjoint. This is followed by the introduction of the scalar and
vector Dirichlet-to-Neumann maps for the monogenic case.  In
Section \ref{sec:f-HT} we construct the ``Vekua-Hilbert transform''
$\Hi_f$ associated to the main Vekua equation in bounded Lipschitz
domains of $\R^3$, and establish some basic facts related to the
elements of its construction.  In Section \ref{sec:Dirichlet-Neumann}
we connect the Vekua-Hilbert transform with the scalar and vector
Dirichlet-to-Neumann maps for the conductivity equation, and verify
the continuous dependence on the boundary values of the conductivity
$f^2$ for the Vekua-Hilbert transform $\Hi_f$ and the quaternionic
Dirichlet-to-Neumann map. In an Appendix we use the
ingredients of the article to present an improved generalized solution
of the div-curl system, removing the previous requirement of
star-shapedness of the domain \cite{DelPor2017}.

\section{Preliminaries\label{sec:preliminaries}}

We follow almost entirely the notation in \cite{GuHaSpr2008} regarding
quaternions and the basic integral operators related to
hyperholomorphic (or monogenic) functions.  In particular $e_0=1$
denotes the multiplicative unit of the non-commutative algebra $\H$ of
quaternions, while the nonscalar units are $e_1,e_2,e_3$ and satisfy
$e_ie_j=-e_je_i$ for $i\not=j$. A quaternion is
$x=x_0+\sum_{i=1}^3{e_i x_i}=\Sc x + \Vec x\in\H$ ($x_i\in\R$,
$\Sc x=x_0$) and we freely identify the subspaces $\Sc\H$, $\Vec\H$
with the real numbers $\R$ and Euclidean space $\R^3$ respectively.
For a domain $\Omega\subseteq\R^3\subseteq\H$ we have function spaces
such as $C^r(\Omega,\H)$ and in particular the Sobolev spaces
\begin{align*} 
  W^{1,p}(\Omega,\H) &= \left\{u\in L^p(\Omega,\H)\colon\
                    \grad u_i\in L^p(\Omega,\R^3)\right\}, \\
  W_0^{1,p}(\Omega,\H) &= \overline{C_0^{\infty}(\Omega,\H)}
                    \subseteq W^{1,p}(\Omega,\H), 
\end{align*}
where $1\leq p\leq \infty$ and $C_0^\infty$ denotes smooth functions
of compact support. Facts about Sobolev spaces are drawn from
\cite{Fournier1978,Brezis2011,Grisvard1985}. The space
\begin{align*} 
W^{1-1/p,p}(\partial\Omega,\H) = 
     \{\varphi\in L^p(\partial\Omega,\H)&\colon\ u|_{\partial\Omega}=
		 \varphi \mbox{ for some } u\in W^{1,p}(\Omega,\H)\},
\end{align*}
of the boundary values of functions in $W^{1,p}(\Omega,\H)$ is
justified by the Trace Theorem, valid for bounded open sets $\Omega$
in $\R^n$, with Lipschitz boundary $\partial\Omega$.  The trace
operator
\begin{align*}
   \tr \colon W^{1,p}(\Omega,\H) \to W^{1-1/p,p}(\partial\Omega,\H),
	 \quad \tr u = u|_{\partial\Omega},
\end{align*} 
is a surjective, bounded linear operator with a continuous right
inverse.  When $p=2$, we will use the usual notation
$H^{1/2}(\partial\Omega,\H)=W^{1-1/2,2}(\partial\Omega,\H)$.  Whenever
$\partial\Omega$ is mentioned we specify the smoothness required for
applying the basic facts about Sobolev spaces. For facility of
notation, we will write $\tr_+ w(\vec x)$ and $\tr_- w(\vec x)$ for
the non-tangential limit of $w(\vec y)$ as $\vec y\in\Omega^{\pm}$
tends to $\vec x\in\partial\Omega$, where $\Omega^+=\Omega$ and
$\Omega^-=\R^3\setminus \overline{\Omega}$. The
above applies to the Sobolev subspaces with $\R$ or $\R^3$ in place of
$\H$. We gather in Theorem \ref{theo:T_acotado} below the facts we
will need about certain integral operators on these spaces.

\subsection{Quaternionic analysis\label{sec:quatanalysis}}

From now on $\vec x\in\R^3\subseteq\H$. The Moisil-Teodorescu (or
Cauchy-Riemann or occasionally, Dirac) differential operator 
\begin{align}\label{eq:operador_Dirac}
  D = e_1\frac{\partial}{\partial{x_1}} +
      e_2\frac{\partial}{\partial{x_2}} + 
      e_3\frac{\partial}{\partial{x_3}}
\end{align}
applied on the left to $w=w_0+\vec w$ gives
\[   Dw =-\div\vec w + \grad w_0 + \curl\vec w.  
\] 
A function $w\in C^1(\Omega,\H)$ is called \textit{monogenic} in 
$\Omega$ when $Dw=0$ and we write
$w\in\M(\Omega)$. Thus $w\in\M(\Omega)$ if and only if
\begin{align}\label{eq:monogenicas_izquierda}
   \div\vec w = 0, \       \curl\vec w =-\grad w_0.
\end{align}  
From $\Delta w_0=-D^2w_0$, we have
$\M(\Omega)\subseteq\Har(\Omega,\H)$. Write $\Sol(\Omega,\R^3)$ and
$\Irr(\Omega,\R^3)$ for the fields with vanishing divergence
(solenoidal) and vanishing curl (irrotational), respectively, and
\begin{align} \label{eq:SI} 
  \SI(\Omega)  = &\Sol(\Omega,\R^3) \cap \Irr(\Omega,\R^3), 
\end{align}
for the solenoidal-irrotational vector fields (vectorial monogenic constants).

The \textit{Cauchy kernel} is the SI vector field
\begin{align} 
    E(\vec x) = -\frac{\vec x}{4\pi|\vec x|^3}, \quad \vec x \in \R^3-\{0\}.
\end{align}
The \textit{Cauchy} operator
\begin{align}\label{eq:operador_Cauchy}
   F_{\partial\Omega}[\varphi](\vec x)=\int_{\partial\Omega} E(\vec y-  
	 \vec x) \eta(\vec y) \varphi(\vec y) \,ds_{\vec y}, \quad 
   \vec x\in\R^3\setminus \partial\Omega,
\end{align}
is related to the \textit{Teodorescu transform}
\begin{align}\label{eq:operador_Teodorescu}
  T_{\Omega}[w](\vec x) = -\int_{\Omega} E(\vec y-\vec x) w(\vec y)
    \,d\vec y, \quad \vec x\in\R^3, 
\end{align}
by the Borel-Pompeiu formula
\begin{align}\label{eq:formula_BP}
   T_{\Omega}[Dw](\vec x)+F_{\partial \Omega}[\tr w](\vec x)= \left\{
   \begin{array}{ll}
      w(\vec x), \ & \vec x\in \Omega,\\
      0,           & \vec x\in \R^3\setminus \overline{\Omega}.
   \end{array}
\right. 
\end{align} 
The Teodorescu transform acts as the right inverse operator of
$D$, $DT_\Omega[w]=w$, valid for quaternionic $w$ with appropriate
continuity suppositions.

The three-dimensional singular Cauchy integral operator
\begin{align}\label{eq:involucion_S}
  S_{\partial \Omega}[\varphi](\vec x)=
  2\PV \int_{\partial \Omega}{E(\vec y-\vec x) \eta(\vec y) 
  \varphi(\vec y) ds_{\vec y}}, \quad \vec x\in \partial \Omega
\end{align}
satisfies $S_{\partial\Omega}^2=I$ and also the Plemelj-Sokhotski formulas
\begin{align}\label{eq:formula_PS}
  \tr_\pm  F_{\partial \Omega}[\varphi](\vec x)
	 = \frac{1}{2}[\pm\varphi(\vec x) + S_{\partial\Omega}[\varphi]
	 (\vec x)];
\end{align} 
from this it is seen that $S_{\partial \Omega}[\varphi]=\varphi$ is 
necessary and sufficient for $\varphi$ to represent the boundary
values of a monogenic function defined in $\Omega$; i.e.\
$\varphi=\tr_+F_{\partial \Omega}[\varphi]$; the opposite condition
$S_{\partial \Omega}[\varphi]=-\varphi$ is necessary and sufficient
for $\varphi$ to have a monogenic continuation into the exterior
domain $\Omega^-$ vanishing at
$\infty$. 

The abovementioned operators are connected with the
\textit{single-layer potential} \cite{Chen1992, McLean2000}
\begin{equation}\label{eq:single_layer_operator} 
    M[\varphi](\vec x) = \int_{\partial\Omega} 
    \frac{\varphi(\vec y )}{4\pi|\vec y -\vec x |} \,ds_{\vec y} ,\quad
    \vec x \in \R^3\setminus \partial\Omega 
\end{equation}
and with the \textit{boundary single-layer operator} $\tr M$ obtained by
evaluating the integral in \eqref{eq:single_layer_operator}
for $x\in\partial\Omega$, thus extending $M$ to all of $\R^3$.

The integral operator \eqref{eq:operador_Teodorescu} makes sense when
$w$ is integrable, as do the operators \eqref{eq:operador_Cauchy},
\eqref{eq:involucion_S} when $\varphi\in L^p(\partial\Omega,\H)$
\cite{GuHaSpr2008}.  According to \cite[Remark 2.5.11]{GuSpr1990}, the
Plemelj-Sokhotski formulas are valid in the Sobolev spaces
$W^{1-1/p,p}(\partial\Omega,\H)\subseteq L^p(\partial\Omega,\H)$.  Let
$H^{-1/2}(\partial\Omega,\H)$ be the dual of the Sobolev space
$H^{1/2}(\partial\Omega,\H)$. We have

\begin{teorema}\label{theo:T_acotado}
Let $\Omega$ be a bounded domain and let $1<p<\infty$. The following 
operators are continuous:
\begin{enumerate}
\item[(a)] The Teodorescu transform \cite[Theorem 8.4]{GuHaSpr2008},
  \cite[Theorem 4.1.7]{GuHaSpr2016}
\[ T_{\Omega} \colon L^p(\Omega,\H) 
	 \to W^{1,p}(\Omega,\H);  \] 
\item[(b)] The singular Cauchy integral operator \cite[p.\ 421]{Mikh1986}
\[ S_{\partial\Omega} \colon W^{1-1/p,p}(\partial\Omega,\H) \to
   W^{1-1/p,p}(\partial\Omega,\H);  \]  
\item[(c)] The single-layer potential \cite[p.\ 38]{Colton1992}
\[ M \colon H^{-1/2}(\partial\Omega,\H) \to W^{1,2}(\Omega,\H); \]  	
\item[(d)] The boundary single-layer operator \cite[Proposition
  2.4.7]{Knu2002}, \cite[Theorem 6.12]{McLean2000}
\[ \tr M\colon H^{-1/2}(\partial\Omega,\H) \to H^{1/2}(\partial\Omega,\H)  .
\]
\end{enumerate}	
\end{teorema}

\subsection{Components of the Cauchy and singular Cauchy integral
  operators\label{subsec:components_F,S}}

Following the notation of the decomposition used in \cite{DelPor2017}
we write
\begin{align}\label{eq:decomposition_T}
   T_{\Omega}[w_0+\vec w] = \Ti[\vec w] + \Tii[w_0]+\Tiii[\vec w],
\end{align}
where
\begin{align}\label{eq:T1_T2_T3}
  \Ti[\vec w](\vec x) &= 
   \int_{\Omega} E(\vec y-\vec x) \cdot \vec w(  \vec y) \,d\vec y, \nonumber\\
  \Tii[w_0](\vec x)    &= 
   -\int_{\Omega} w_0(\vec y) E(\vec y-\vec x)   \,d\vec y,  \nonumber\\
  \Tiii[\vec w](\vec x) &= 
   -\int_{\Omega} E(\vec y-\vec x) \times \vec  w(\vec y) \,d\vec y.
\end{align}

In a similar way we give a decomposition of the Cauchy operator 
\cite[Theorem 2.5.5]{GuSpr1990}
\begin{align}\label{eq:Cauchy-operator}  
 F_{\partial\Omega} \colon W^{1-1/p,p}(\partial\Omega,\H) 
	\longrightarrow  W^{1,p}(\Omega,\H)\cap \M(\Omega).	
\end{align} 
Let $\eta$ denote the unit normal vector to $\partial\Omega$.  For
real-valued functions $\varphi_0\in L^p(\partial\Omega,\R)$ we
decompose
$F_{\partial\Omega}[\varphi_0] = \Fi[\varphi_0] + \Fii[\varphi_0]$
into the normal and tangential components
\begin{align}\label{eq:F0_F1}
  \Fi[\varphi_0](\vec x) &= - \int_{\partial\Omega} E(\vec y-\vec x) 
	\cdot \eta(\vec y)\, \varphi_0(\vec y) \,ds_{\vec y}, \nonumber\\
  \Fii[\varphi_0](\vec x) &= \int_{\partial\Omega} E(\vec y-\vec x) 
	\times \eta   (\vec y)\, \varphi_0(\vec y) \,ds_{\vec y}
\end{align}
for $\vec x\in\R^3\setminus \partial\Omega$.  Analogously to
\cite[Proposition 3.2]{DelPor2017} for $T_\Omega$, the components of
$F_{\partial\Omega}$ can be expressed in terms of the single-layer
potential $M$ of \eqref{eq:single_layer_operator},
\begin{align}\label{eq:F0_F1-div-curl} 
  \Fi[\varphi_0]  =   \nabla \cdot M[\varphi_0 \eta],\quad
  \Fii[\varphi_0] = - \nabla \times M[\varphi_0\eta],
\end{align}
where $\Fii[\varphi_0]\in\Sol(\Omega,\R^3)$ and
$\curl \Fii[\varphi_0]\in\Irr(\Omega, \R^3)$. More generally, for
every $\varphi\in W^{1,p}(\Omega,\H)$ we have that
$F_{\partial\Omega}[\varphi] = - D M[\eta\varphi]$ (\cite[Proposition
2.5.3]{GuSpr1990}, note the change of sign).

Similarly, we can decompose $S_{\partial\Omega}=K_0+\vecK$, where the
component operators are
\begin{align}
   K_0[\varphi](\vec x) &= 2 \PV \int_{\partial\Omega}{ 
   -E(\vec y-\vec x) \cdot \eta(\vec y) \, \varphi(\vec y) \,ds_{\vec y
   }}, \nonumber\\
   \vecK[\varphi](\vec x) &=2 \PV \int_{\partial\Omega} E(\vec y
	 -\vec x) \times \eta(\vec y) \, \varphi(\vec y)  \,ds_{\vec y} 
    = \sum_{k=1}^3 e_iK_i[\varphi](\vec x), \label{eq:operadores_K_i}
\end{align}
with $K_i[\varphi]$ ($i=1,2,3$) having as integration kernel the $i$th
quaternionic component $[E(\vec y-\vec x) \times \eta(\vec
y)]_i$.
Note that $S_{\partial\Omega}$ is a right $\H$-linear operator, and in
particular for real-valued functions $\varphi_0$,
$\Sc S_{\partial\Omega}[\varphi_0]=K_0[\varphi_0]$ and
$\Vec S_{\partial\Omega}[\varphi_0] =\vecK[\varphi_0]$. We will
frequently use the fact that since a scalar constant $c_0\in\R$ is
monogenic,  $S_{\partial\Omega}[c_0]=c_0$, so
\begin{equation}  \label{eq:Kconstant}
 K_0[c_0] = c_0,\quad \vecK[c_0]=0. 
\end{equation}

The operators $K_0$ and $\vecK$ \eqref{eq:operadores_K_i} acting on
$L^p(\partial\Omega,\H)$ and $L^p(\partial\Omega,\R^3)$ respectively
have as adjoints
\begin{align*}
   K_0^*[\varphi](\vec x) &= 2 \PV \int_{\partial\Omega}{ 
   E(\vec y-\vec x) \cdot \eta(\vec x) \, \varphi(\vec y) \,ds_{\vec y
   }}, \nonumber\\
   \vecK^*[\vec \varphi](\vec x) &=2 \PV \int_{\partial\Omega} 
	 -E(\vec y-\vec x) \times \eta(\vec x) \cdot \vec\varphi(\vec y)  
	 \,ds_{\vec y}=\sum_{i=1}^3{K_i^*[\varphi_i](\vec x)},
\end{align*}
on $L^q(\partial\Omega,\H)$ and $L^q(\partial\Omega,\R^3)$,
respectively, where the duality pairing of
$\H$-valued functions is $\Sc \int_{\partial\Omega} \overline{\varphi(\vec y)} \psi(\vec y)\,ds_{\vec y}$ and $1/p+1/q=1$.

Let $A$ denote the boundary averaging operator
\begin{align}\label{eq:averaging-operator}
   A[\varphi]= \frac{1}{\sigma_\Omega}\int_{\partial\Omega}
		{\varphi(\vec y) \, ds_{\vec y}}
\end{align}
(with $\sigma_\Omega$ chosen so that $A[1]=1$), which induces a
natural mapping $I-A$ from $L^p(\partial\Omega,\H)$ to
$L_0^p(\partial\Omega,\H)$, where $L_0^p(\cdot)$ is the
subspace of functions in $L^p(\cdot)$ with mean $0$. 

The operator $K_0$ has been thoroughly studied due to its importance
in solving the Dirichlet Problem, and has very good properties
\cite{DahKen1987,Kenig1994}; for example on a $C^{1,\gamma}$ $(\gamma>0)$ domain
\begin{align*}
    \left|E(\vec y-\vec x) \cdot \eta(\vec y)\right| 
    \leq \frac{C}{|\vec y-\vec x|^{2-\gamma}},
\end{align*}
and thus $K_0$ is a compact operator from $L^p(\partial\Omega)$ to
itself $(1<p<\infty)$. Likewise $\vecK$ is bounded from
$L^p(\partial\Omega,\R)$ to itself and from
$W^{1-1/p,p}(\partial\Omega,\R)$ to itself, because
$S_{\partial\Omega}$ is bounded in $L^p(\partial\Omega,\H)$
\cite[Theorem 2.5.8]{GuSpr1990} and in
$W^{1-1/p,p}(\partial\Omega,\H)$ (see Theorem
\ref{theo:T_acotado}(b)), respectively. We will always assume that
the complement of $\Omega$ is connected.  When $\Omega$ is a bounded
Lipschitz domain, although Fredholm theory is not applicable, it is
possible to verify the invertibility of $I+K_0$. We summarize here the
results on self mappings that we will need.

\begin{proposicion}\cite{DahKen1987,Kenig1994}
\label{prop:invertibility-I+K_0}
There is $\epsilon(\Omega)$, depending only on the Lipschitz character of $\partial\Omega$, such that
\begin{enumerate}
\item[(a)] If $\partial\Omega$ is Lipschitz, $2-\epsilon(\Omega)<p<
\infty$, then $I+K_0$ is invertible on $L^p(\partial\Omega)$ with
bounded inverse.
	 \item[(b)] If $\partial\Omega$ is Lipschitz, $1<p<
2+\epsilon(\Omega)$, then $I+K_0$ is invertible on $W^{1-1/p,p}
(\partial\Omega)$ with bounded inverse.
	 \item[(c)] If $\Omega$ is $C^{1,\gamma}$ Lipschitz for some
$\gamma >0$, $1<p<\infty$, then $I+K_0$ is invertible both on $L^p(
	 \partial\Omega)$ and $W^{1-1/p,p}(\partial\Omega)$ with
bounded inverse.
	 \item[(d)] If $\partial\Omega$ is Lipschitz, $1<q<
2+\epsilon(\Omega)$, or $C^{1,\gamma}$ and $1<q<\infty$, then
$I-K_0^*$ is invertible on $L_0^q(\partial\Omega)$ with bounded
inverse.
\end{enumerate}
\end{proposicion}

\begin{proof}
Part (a) was established in \cite[Theorem 3.1]{Ver1984}
for $p=2$, and then in \cite[Theorem 4.17]{DahKen1987} it was extended
for $2-\epsilon(\Omega)<p<\infty$.

To prove (b), let $1<p<2+\epsilon(\Omega)$. In the proof of
\cite[Theorem 3.3]{Ver1984} it is shown that
$I+K_0=M(I+K_0^*)M^{-1}$. We have noted previously that $I+K_0$ is
bounded on $L^p(\partial\Omega)$ and has a bounded inverse. This fact
is not sufficient for our purpose, but by the same reference
\cite[Theorems 4.17, 4.18]{DahKen1987}, the single layer potential
$M\colon L^p(\partial\Omega)\to W^{1-1/p,p}(\partial\Omega)$ and
$I+K_0^*$ are bounded and have bounded inverses. From this follows the
boundedness of $(I+K_0)^{-1}$ in $W^{1-1/p,p}(\partial\Omega)$.

Parts (c) and (d) were stated in \cite[p.\ 52]{Kenig1994} and
\cite[Theorem 4.17]{DahKen1987}, respectively (see also \cite[Theorem 3.3]{Ver1984}
for $q=2$).
\end{proof}

\begin{proposicion}\label{prop:characterization-trT0-trT2}
 On the Sobolev space $W^{1,2}(\Omega,\R)$,
\[ 2\tr\Ti\circ\nabla=(I-K_0)\circ\tr, \quad
   2\tr\Tiii\circ\nabla=-\vecK\circ\tr.
\] 
\end{proposicion}

\begin{proof}
  Let $w_0\in W^{1,2}(\Omega,\R)$, $\varphi_0=\tr w_0$.  Apply
  \eqref{eq:formula_BP} to $w_0$ and take the trace, and then apply
  \eqref{eq:formula_PS}:
\begin{align*} 
 \tr T_\Omega[\nabla w_0] =&  \varphi_0-\tr F_{\partial\Omega}[\varphi_0]
   =   \varphi_0-\frac{1}{2}(\varphi_0 + S_{\partial\Omega}[\varphi_0]) \\
    =& \frac{1}{2}(I - K_0 - \vecK)[\varphi_0].   
  \end{align*}
Now take the scalar and vector parts.
\end{proof}

\section{Hilbert transform for monogenic
  functions}\label{sec:Hilbert-monogenic}

Before entering on the investigation of the Vekua equation in domains
in $\R^3$, we begin our study of the Hilbert transform in the much
simpler case of monogenic functions of three variables. This refers to
a linear operator which produces the boundary values of the vector
part of a monogenic function, given the boundary values of the scalar
part, thus generalizing the classical operator defined by D.\ Hilbert
for the unit disk or upper half plane in $\C$. This problem has been
studied in the context of Clifford algebras for the unit sphere in
$\R^n$ in \cite{Qian2009,Brac2006-2} and for $k$-forms in Lipschitz
domains in \cite{AKQ2009}.
 
\subsection{Definition of $\Hi$}\label{subsec:Tao-Qian-paper}

From now on $\Omega$ will be a $C^{1,\gamma}$ bounded Lipschitz domain
with connected boundary, $\gamma>0$ and $1<p<\infty$ or $\Omega$ will
be a bounded Lipschitz domain $2-\epsilon(\Omega)<p<\infty$ (unless
another range of $p$ be specified).  Then the operators $K_0$, $\vecK$
and $(I+K_0)^{-1}$ are all bounded from $L^p(\partial\Omega)$ to
$L^p(\partial\Omega)$.

We recall the construction which was given in \cite{Qian2008,Qian2009}
for bounded Lipschitz domains and for the unit ball in $\R^n$.
Specifically when $n=3$, the \textit{Hilbert transform}
\[ \Hi\colon L^p(\partial\Omega,\R)\to L^p(\partial\Omega,\R^3)\]
is defined as
\begin{align}\label{eq:Hilbert_transform}
  \Hi[\varphi_0]  =\vecK(I+K_0)^{-1}\varphi_0 =\frac{1}{2}\vecK[h_0]
\end{align}
with $K_0$, $\vecK$ given in \eqref{eq:operadores_K_i}, and
$h_0=2(I+K_0)^{-1}\varphi_0$.  By the Plemelj-Sokhotski formula
\eqref{eq:formula_PS}, the non-tangential boundary limits of
$F_{\partial\Omega}[h_0]$ exist, and since $h_0$ is $\R$-valued, for
$\vec x\in \partial\Omega$ we have
\begin{align}
  \tr_+ F_{\partial\Omega}[h_0](\vec x)&=
 \frac{1}{2}\left(h_0(\vec x)+\Sc(S_{\partial\Omega}[h_0])(\vec x)\right)
  +\frac{1}{2}\Vec(S_{\partial\Omega}[h_0])(\vec x) \nonumber\\
   &=\frac{1}{2}(I+K_0)h_0(\vec x)+\Hi[\varphi_0](\vec x) \nonumber\\
  &=\left(\varphi_0+\Hi[\varphi_0]\right)(\vec x).\label{eq:u+Hu}
\end{align}
Thus $\varphi_0+\Hi[\varphi_0]$ is the boundary value of the monogenic
function $F_{\partial\Omega}[h_0]$ in $\Omega$, which justifies
calling $\Hi$ a Hilbert transform. The image of the Hilbert transform
$\Hi$ belongs to the space of boundary functions whose harmonic
extension is divergence free because from
\eqref{eq:monogenicas_izquierda} and the construction
\eqref{eq:Hilbert_transform}, the vector part of the monogenic
extension
$W=F_{\partial\Omega}[h_0]=F_{\partial\Omega}[2(I+K_0)^{-1}\varphi_0]$
satisfies $\div \vec W=0$.

From Proposition \ref{prop:characterization-trT0-trT2} observe that the
identity $2\Hi[\varphi_0]=\Hi[(I+K_0)\varphi_0]+\Hi[(I-K_0)\varphi_0]$
can now be expressed as
\begin{equation}  \label{eq:HT2T0}
  \Hi[\varphi_0] =  -\tr \Tiii[\nabla w_0] + \Hi[\tr \Ti[\nabla w_0] ].
\end{equation}
 
\subsection{Properties of $\Hi$ and its adjoint and
  inverse}\label{subsec:inverseHf}

We derive some basic facts of the Hilbert transform $\Hi$, as well as
for the adjoint and a left inverse of $\Hi$. At the end of this
subsection we will see that $\Hi$ belongs to the class of
semi-Fredholm operators.

The Hilbert operator $\Hi$ is a bounded and non-compact operator in
the $L^p$ norm. The boundedness was proved for the ball in \cite[Theorem
6]{Qian2009} and for Lipschitz domains in \cite[Theorem
3.2]{Qian2008}. If $\Hi$ were compact, then $\vecK$ would also be
compact, since $I+K_0$ is bounded on $L^p(\partial\Omega,\R)$. But
since $K_0$ is compact \cite[Cor.\ 2.2.14]{Kenig1994} on $C^1$
domains, $S_{\partial\Omega}$ would then be compact by the
decomposition \eqref{eq:operadores_K_i}, and then
$S_{\partial\Omega }^2=I$ would also be compact, which is absurd.

When we restrict the domain of the Hilbert transform $\Hi$ to
Sobolev space,
the property of boundedness is preserved.  Recall the value
$\epsilon(\Omega)$ discussed in Proposition 
\ref{prop:invertibility-I+K_0}.

\begin{teorema}\label{theo:H_bounded_noncompact}
Let $\Omega$ be a bounded Lipschitz domain. The restriction 
\begin{align*}
	 \Hi \colon W^{1-1/p,p}(\partial\Omega,\R)\to W^{1-1/p,p}(
	 \partial\Omega,\R^3),
\end{align*}  
of the Hilbert transform $\Hi$ is a bounded operator when
$1<p<2+\epsilon(\Omega)$, and also when $1<p<\infty$ and $\Omega$ is a
$C^{1,\gamma}$ Lipschitz domain, $\gamma>0$.
\end{teorema}

\begin{proof}
  We have noted that $\vec K$ is bounded, so the statement follows
  from \eqref{eq:Hilbert_transform} and Proposition
  \ref{prop:invertibility-I+K_0}, parts (b) and (c).
\end{proof}
  
From this it is straightforward to obtain the explicit form of the
adjoint of $\Hi$. Write
$\epsilon^{\pm}(\Omega)=(2\pm \epsilon(\Omega))/(1\pm
\epsilon(\Omega))$.

\begin{proposicion}\label{prop:H_bounded_Lipschitz}
  Let $\Omega$ be a bounded Lipschitz domain. Then the adjoint 
	$\Hi^* \colon L^q(\partial\Omega, \R^3)\to L^q(\partial\Omega, \R)$ 
\begin{align}\label{eq:adjoint-Hilbert-transform}
    \Hi^*[\vec \varphi] = (I+K_0^*)^{-1}\vecK^*[\vec \varphi] 
\end{align}
is bounded on $W^{1-1/q,q}(\partial\Omega)$ for
$\epsilon^+(\Omega)<q<\infty$ and on $L^p(\partial\Omega)$ for
$1<q<\epsilon^-(\Omega)$. When $\Omega$ has $C^{1,\gamma}$ boundary,
$\gamma>0$,
$\Hi^* \colon L^q(\partial\Omega, \R^3)\to L^q(\partial\Omega, \R)$ is
bounded for $1<q<\infty$.
\end{proposicion}
 
We now discuss the invertibility of $\Hi$. The identity
$S_{\partial\Omega}^2=I$ combined with \eqref{eq:operadores_K_i}, when
applied to real-valued functions, produces the identities 
\begin{align}\label{eq:sistema_involucion}
    I-K_0^2=-\sum_{i=1}^3{K_i^2}, 
\end{align}
and $K_0K_i+K_iK_0+K_jK_k-K_kK_j=0$ for
$(i,j,k)=(1,2,3),\ (2,3,1),\ (3,1,2)$.  The equation
\eqref{eq:sistema_involucion} will be particularly useful; the last
three play a similar role to the commutative relations enjoyed by the
Riesz transforms $R_i$ $(i=1,2,3)$ in a half space of $\R^3$ \cite[p.\
91]{KravShap1996}.

 In \cite{Qian2008, Qian2009} reference is
made to the inverses of $I\pm K_0$ (see also \cite{Kenig1994}).
However, we observe the following.

\begin{proposicion}\label{prop:invertible}
  Let $\Omega$ be Lipschitz and $\epsilon^+(\Omega)<p<\infty$ or
  $C^{1,\gamma}$ Lipschitz and $1<p<\infty$. Then $\Ker(I-K_0)= \R$ on
  $L^p(\partial\Omega)$.
\end{proposicion}

\begin{proof}
  Let $c_0\in\R$. Then \eqref{eq:Kconstant} shows that
  $c_0\in\Ker(I-K_0)$. We now verify that the only elements of
  $\Ker(I-K_0)$ are constants. Since the adjoint $I-K_0^*$ is
  invertible in $L_0^q(\partial\Omega)$ by Proposition 
	\ref{prop:invertibility-I+K_0}(d), it follows from the Banach 
	Closed Range Theorem that the image is $\Im{(I-K_0)}=L_0^p(\partial\Omega)$.  Thus
  $\Ker(I-K_0)|_{L_0^p(\partial\Omega)}=\{0\}$. Finally, let
  $g\in L^p(\partial\Omega)$ such that $g\in \Ker (I-K_0)$. Let
  $f=(I-A)g$, where $A$ is the boundary averaging operator 
	\eqref{eq:averaging-operator}. Then by \eqref{eq:Kconstant},
\begin{align*}
   (I-K_0)f = f - (K_0[g]-K_0[Ag]) = f - (g-Ag) =0.
\end{align*}
Since $f\in L_0^p(\partial\Omega,\R)$, we have $f=0$; that is,
$g=A[g]\in \R$.
\end{proof}

Note also that $K_0$ does not interfere with the averaging process:
$AK_0[\varphi_0]=A[\varphi_0]$, because
$2 \PV \int_{\partial\Omega}{E(\vec y-\vec x)\cdot \eta(\vec x) \, ds_{\vec y}}=1$.
For this reason and by Proposition \ref{prop:invertible}, the operator $I-K_0$ sends $L_0^p(\partial\Omega,\R)$ to itself, and has an inverse
\[   (I-K_0)^{-1}\colon L_0^p(\partial\Omega,\R) \to L_0^p(\partial\Omega,\R)
\]
with $\epsilon^+(\Omega)<p<\infty$ when $\Omega$ is Lipschitz and $1<p<\infty$ when $\Omega$ is $C^{1,\gamma}$.
  
We define the operator
$\G\colon L^{p}(\partial\Omega,\R^3) \to L_0^{p}(\partial\Omega,\R)$ by
\begin{align} \label{eq:transformada_Hilbert_monogenica_vectorial}
 \G[\vec\varphi] = -(I-K_0)^{-1} (I-A)\vecK \cdot \vec\varphi.
\end{align}
We have used the notational convention
\begin{align*}
   T\cdot \varphi =\sum_{i=0}^3{T_i \varphi_i}
\end{align*} 
which we will use whenever $T=\sum_{i=0}^3{e_iT_i}$ where $T_i$ are
right $\H$-linear operators which send scalar-valued functions to
scalar-valued functions, and $\varphi=\sum_{i=0}^3{e_i \varphi_i}$
with $\varphi_i$ scalar-valued.

\begin{proposicion}\label{prop:inverse-Hilbert}
  Assume that $\Omega, p$ satisfy the hypotheses of Proposition
  \ref{prop:invertible}. Then $\G$ is a left inverse for the Hilbert
  transform $\Hi$ on $L_0^p(\partial\Omega,\R)$.
\end{proposicion}

\begin{proof}
  Let $\varphi_0\in L_0^{p}(\partial\Omega,\R)$.  By
  \eqref{eq:Hilbert_transform} and \eqref{eq:sistema_involucion},
\begin{align*}
  \G \circ \Hi[\varphi_0]
  &= -(I-K_0)^{-1}(I-A)(\vecK \cdot 
    \vecK) (I+K_0)^{-1}\varphi_0\\
  &=(I-K_0)^{-1}(I-A)\left(-\sum_{i=1}^3{K_i^2}\right)(I+K_0)^{-1}\varphi_0\\
  &=(I-K_0)^{-1}(I-A)(I-K_0^2)(I+K_0)^{-1}\varphi_0\\
  &=(I-K_0)^{-1}(I-K_0+AK_0-A)\varphi_0\\
	&=\varphi_0,
\end{align*}\
where the last equality uses $AK_0=A$.
\end{proof}

The proof of the non-compactness of $\Hi$ fails in the case of bounded
Lipschitz domains because $K_0$ need not be compact
\cite{FaJoLe1977}. However, the existence of its left inverse
automatically guarantees the non-compactness. Other straightforward
consequences are the following.

\begin{corolario}
Under the same hypotheses,
\begin{enumerate}
\item[(a)] Restricted to $L_0^p(\partial\Omega,\R)$, the Hilbert
  transform $\Hi$ is injective and its left inverse $\G\colon L^{p}(\partial\Omega,\R^3)\to
  L_0^{p}(\partial\Omega,\R)$ is surjective.
\item[(b)] The left inverse $\G$ of the Hilbert transform is a bounded
  and non-compact operator.
\end{enumerate}
\end{corolario}

From \eqref{eq:transformada_Hilbert_monogenica_vectorial} and  
$A^*=A$, the adjoint operator
$\G^*\colon L_0^p(\partial\Omega,\R)\to L^p(\partial\Omega,\R^3)$ is
given by
\begin{align*}
   \G^*[\varphi_0]=-\sum_{i=1}^3{e_iK_i^*}(I-A)(I-K_0^*)^{-1}
	 [\varphi_0]=-\sum_{i=1}^3{e_iK_i^*}(I-K_0^*)^{-1}[\varphi_0].
\end{align*}

We now look at the question of the images under $\G$ of the boundary
values of SI vector fields. Write $\SI(\partial\Omega)$ for the space
of boundary values of SI vector fields in $\Omega$ which extend to
$\overline{\Omega}$, which we recall from \eqref{eq:SI} are the purely
vectorial monogenic constants. Since SI vector fields are harmonic,
the SI extension of $\vec \varphi\in\SI(\partial\Omega)$ to the
interior is unique. The elements of
$\SI(\partial\Omega)$ are annihilated by $\G$, more precisely
\[  \SI(\partial\Omega)\cap
    L^{p}(\partial\Omega,\R^3) \subseteq \Ker\G.
\]   	
Because for every $\vec\varphi\in\SI(\partial\Omega)$,
$S_{\partial\Omega}[\vec \varphi]=\vec \varphi$, so
$\vecK\cdot\vec\varphi=0$. By
\eqref{eq:transformada_Hilbert_monogenica_vectorial},
$\vec\varphi\in\Ker\G$.

Clearly $\Ker\Hi=\R$ since the only scalar-valued monogenic functions
are constants. One important fact about $\Im \Hi$ is
$\SI(\partial\Omega)\cap
L^{p}(\partial\Omega,\R^3)\cap\operatorname{Im}\Hi=\{\vec 0\}$;
moreover,
 
\begin{corolario}\label{prop:H-semi-Fredholm}
  Under the same hypotheses, the Hilbert
  transform $\Hi$ on $L^p(\partial\Omega,\R)$ is a left semi-Fredholm
  operator.
\end{corolario}

\begin{proof}
  It is enough to prove that when the domain of $\Hi$ is restricted to
  $L_0^p(\partial\Omega,\R)$, the image $\Im\Hi$ is closed in
  $L^p(\partial\Omega,\R^3)$ \cite[Chapter 5]{Kub2012}. By Proposition
  \ref{prop:inverse-Hilbert}, $\G^* \circ \Hi^*=I$, so $\Hi^*$ is 
	surjective. As a consequence of the Banach Closed Range theorem, $\Hi$ 
	has closed range.
\end{proof}

Since $\R=\Ker \Hi=\Ker \vecK$ and $\Im \Hi=\Im \vecK$, the vector operator
$\vecK$ is also left semi-Fredholm.

\subsection{Dirichlet-to-Neumann map}\label{subsec:D-N-monogenic}

Intimately related to the Hilbert transform is the
Dirichlet-to-Neumann (D-N) operator \cite{Behrndt2015}, which plays a
fundamental role in the study of elliptic partial differential
equations.  In the rest of this article we restrict to the case $p=2$
and work in domains $\Omega$ with Lipschitz boundary.

The following Hilbert spaces associated with the operators $\div$ and $\curl$
 appear in many electromagnetism problems. Following \cite[Chapter
9]{Dautray1985} and \cite[Chapter 1]{Girault1986}, let 
\begin{align*}
   W^{2,\div}(\Omega,\R^3) &= \left\{\vec u\in L^2(\Omega,\R^3)\colon\
                    \div \vec u\in L^2(\Omega,\R)\right\}, \\
   W^{2,\curl}(\Omega,\R^3) &= \left\{\vec u\in L^2(\Omega,\R^3)\colon\
                    \curl \vec u\in L^2(\Omega,\R^3)\right\},
\end{align*}
with norms $\|\vec u\|_{L^2}+\|\div\vec u\|_{L^2}$ and
$\|\vec u\|_{L^2}+\|\curl\vec u\|_{L^2}$ respectively.  Observe that
the conditions defining these spaces are weaker than requiring
$\grad u$ to be in $L^2$.  Therefore
$W^{1,2}(\Omega,\R^3)\subset W^{2,\div}(\Omega,\R^3)\cap
W^{2,\curl}(\Omega,\R^3)$,
but for the opposite containment it is necessary to add certain
boundary conditions; see Proposition
\ref{prop:Friedrichs-inequalities} below for the required
constraints. This result is sometimes enunciated as Friedrichs'
inequality; references include \cite{ABDG1998,Kri1984,Saranen1982}.

The \textit{normal and tangential trace operators} \cite{Dautray1985}
\begin{align*}
   \gamma\nrm(\vec u)=\vec u|_{\partial\Omega} \cdot \eta, \quad
   \gamma\tng(\vec u)=\vec u|_{\partial\Omega}\times \eta.
\end{align*}
are defined on $W^{2,\div}(\Omega,\R^3)$ and $W^{2,\curl}(\Omega,\R^3)$
respectively. They are weakly defined as 
\begin{align}\label{eq:normal-tangential-trace}
   \nonumber \langle \gamma\nrm(\vec u), \tr v_0 \rangle_
	 {\partial\Omega}&=\int_{\Omega}{\vec u \cdot \nabla \psi_0 \, 
	 d\vec y}+\int_{\Omega}{\div \vec u v_0 \, d\vec y},\\
	 \langle \gamma\tng (\vec u), \tr \vec v \rangle_{\partial\Omega}&=
		\int_{\Omega}{\vec u \cdot \curl \vec v \, d\vec y}
		-\int_{\Omega}{\curl\vec u \cdot \vec v \, d\vec y},
\end{align}
for every $v=v_0+\vec v\in W^{1,2}(\Omega,\R^3)$. Let
$W_0^{2,\div}(\Omega,\R^3)$ and $W_0^{2,\curl}(\Omega,\R^3)$ be the
kernels of the trace operators $\gamma\nrm$ and $\gamma\tng$,
respectively.

\begin{proposicion}(\textit{Friedrichs' inequalities} \cite[Chapter 9,
  Corollary 1]{Dautray1985}, \cite[Remark 2.14]{ABDG1998})
\label{prop:Friedrichs-inequalities}
Let $\Omega$ be a $C^{1,1}$ bounded Lipschitz domain. If  respectively
$\gamma\nrm(\vec u)\in H^{1/2}(\partial\Omega,\R)$ or
$\gamma\tng(\vec u)\in H^{1/2}(\partial\Omega,\R^3)$, then
\begin{align*}
   \|\vec u\|_{W^{1,2}}^2 &\leq C\left(\|\vec u\|_{L^2}^2 + 
	 \|\curl \vec u\|_{L^2}^2 + \|\div \vec u\|_{L^2}^2 + 
	 \|\gamma\nrm(\vec u)\|_{H^{1/2}}^2\right) 
\end{align*}
or	
\begin{align*}	 \|\vec u\|_{W^{1,2}}^2 &\leq C\left(\|\vec u\|_{L^2}^2 + 
	 \|\curl \vec u\|_{L^2}^2 + \|\div \vec u\|_{L^2}^2 + 
	 \|\gamma\tng(\vec u)\|_{H^{1/2}}^2\right),
\end{align*}
respectively, where $C>0$ only depends on $\partial\Omega$. 
\end{proposicion}

Let 
\begin{align}\label{eq:kernels}
   \nonumber W\nrm^{2,\div\text{-}\curl}(\Omega,\R^3)=W_0^{2,\div}
	 (\Omega,\R^3)\cap W^{2,\curl}(\Omega,\R^3) ,\\
    W\tng^{2,\div\text{-}\curl}(\Omega,\R^3)=W^{2,\div}(\Omega,\R^3)\cap 
		W_0^{2,\curl}(\Omega,\R^3),
\end{align}
with the norm $\|\vec u\|_{L^2}^2+\|\div \vec u\|_{L^2}^2+\|\curl \vec u\|_{L^2}^2$.

\begin{proposicion}{\rm \cite[Theorems 2.8, 2.9, 2.12]{ABDG1998}, 
\cite{Weber1980}.}
\label{prop:div-curl-spaces}
\\(a) Let $\Omega$ be a bounded $C^{1,1}$ Lipschitz domain. Then
$ W\nrm^{2,\div\text{-}\curl}(\Omega,\R^3)$ and  $W\tng^{2,\div\text{-}\curl}(\Omega,\R^3)$ are contained in $W^{1,2}(\Omega,\R^3)$. 
\\(b) The inclusions of $W\nrm^{2,\div\text{-}\curl}(\Omega,\R^3)$ and $W\tng^{2,\div\text{-}\curl}(\Omega,\R^3)$ into $L^2(\Omega,\R^3)$ are compact operators.
\end{proposicion}

\begin{proposicion}\label{prop:monogenic constants}{\rm \cite[Chapter 9, Corollary 2]{Dautray1985}.}
  Let $\Omega$ be a bounded $C^{1,1}$ Lipschitz domain. The subspaces
  of ``normally'' and ``tangentially'' monogenic constants in the
  interior $\Omega=\Omega^+$ or exterior domain $\Omega^-$,
\begin{align} \label{eq:bdry_constant_monogenics_space}
  \SI\nrm(\Omega^{\pm}) &= \left\{\vec u\in L^2(\Omega^{\pm},\R^3)\colon\
    \div \vec u=0, \curl \vec u=0, \gamma\nrm(\vec u)=0\right\},
  \nonumber\\
  \SI\tng(\Omega^{\pm}) &= \left\{\vec u\in L^2(\Omega^{\pm},\R^3)\colon\
   \div \vec u=0, \curl \vec u=0, \gamma\tng(\vec u)=0\right\},
\end{align}
have finite dimension.  
\end{proposicion}

With these preliminaries we introduce the ``quaternionic
Dirichlet-to-Neumann map''
\begin{align}\label{eq:quaternionic-D-N-monogenic}
   \nonumber  \Lambda\colon H^{1/2}(\partial\Omega,\R)&
   \to H^{-1/2}(\partial\Omega,\H)\\
   \varphi_0 &\mapsto (D w_0|_{\partial\Omega}) \eta,
\end{align}
where $w_0\in W^{1,2}(\Omega,\R)$ is the unique harmonic extension of
$\varphi_0$; note the quaternionic multiplication of vectors. 
 
Since
$\Sc((D w_0|_{\partial\Omega}\eta)\overline{v})=-(\gamma\nrm\nabla
w_0)v_0+(\gamma\tng\nabla w_0)\cdot\vec v$
for every $v=v_0+\vec v\in W^{1,2}(\Omega,\H)$, by the weak
definitions of $\gamma\nrm$ and $\gamma\tng$
\eqref{eq:normal-tangential-trace}, we have
\begin{align}\label{weak-monogenic-D-N}
  \langle \Lambda[\varphi_0], \tr v \rangle_{\partial\Omega}
 =&\int_\Omega{\nabla w_0 \cdot (-\nabla v_0+\curl\vec v) \,d\vec y } \nonumber \\
 =& \Sc \int_{\Omega}\nabla w_0 \Dr v \, d\vec y,
\end{align}
where we write $\Dr$ for the right-sided operator 
\[ \Dr[v]=vD = \sum_1^3(\partial_iv)e_i =
   -\div\vec v+(\nabla v_0-\curl\vec v).
\] 
 The scalar and vector parts of the quaternionic product 
$(D w_0|_{\partial\Omega}) \eta$ give
\begin{align*}
   \Lambda[\varphi_0] &=\Lambda_0[\varphi_0]+\vec \Lambda[\varphi_0] 
\end{align*}
with $\Lambda_0[\varphi_0]=-\gamma\nrm\nabla w_0$,
$\vec\Lambda[\varphi_0]=\gamma\tng\nabla w_0$. Thus the scalar part of
$\Lambda[\varphi_0]$ coincides with the negative of the usual scalar
D-N map for the Laplacian Dirichlet problem
\cite{DahKen1987,Kenig1994}.  We will verify in subsection
\ref{subsec:quatDNmap} that $\Lambda[\varphi_0]$ does indeed lie in
$H^{-1/2}(\partial\Omega,\H)$ as implied by \eqref{eq:quaternionic-D-N-monogenic}.
 
(In the two-dimensional context, such as in \cite{AP2006}, one has
only a scalar D-N mapping, denoted commonly by ``$\Lambda$''.)

As usual $W^{2,2}(\Omega,\R)$ is the Sobolev space of scalar functions
whose gradient belongs to $W^{1,2}(\Omega,\R)$ and $H^{3/2}(\partial\Omega,\R)$ is the space of boundary values of functions in $W^{2,2}(\Omega,\R)$.

\begin{proposicion} \label{prop:T=-MLambdatr}
 $T_\Omega\circ \nabla=-M\circ\Lambda\circ\tr$ on $\Har(\Omega,\R)\cap W^{2,2}(\Omega,\R)$.
\end{proposicion}
\begin{proof} In \cite[Propositions 4.7,
  4.8]{DelPor2017} it was seen that
  $\Ti[\vec w]=M[\vec w|_{\partial\Omega} \cdot\eta]$ for all
  $\vec w\in \Sol(\overline{\Omega},\R^3)$ and
  $\Tiii[\vec w]=-M[\vec w|_{\partial\Omega} \times\eta]$
  for all $\vec w\in\Irr(\overline{\Omega},\R^3)$, where $M$ is the
  single-layer operator \eqref{eq:single_layer_operator}. 
\end{proof}

Note that by \eqref{eq:single_layer_operator}, $M$ is a scalar operator,
so 
\begin{equation}  \label{eq:TMLambda}
  \Ti\circ\nabla = -M\circ\Lambda_0\circ\tr,\quad
  \Tiii\circ\nabla = -M\circ\vec\Lambda\circ\tr.
\end{equation}

We proved that the Hilbert transform $\Hi$ is a non-compact operator. However, when restricted to $\Ker\vec\Lambda$, by dimensional
properties of $\SI\tng(\Omega)$, then $\Hi$ becomes compact.
Recall that we are always assuming that $\partial\Omega$ is connected.

\begin{proposicion}
  $\Ker \Lambda_0=\R$, $\dim \Ker \vec\Lambda<\infty$, and
  $\Ker \vec\Lambda=\R$.
\end{proposicion}

\begin{proof}
Let $\varphi_0\in H^{1/2}(\partial\Omega,\R)$, and let $w_0$ be its harmonic extension. If $\varphi_0\in \Ker \Lambda_0$, then $w_0$ satisfies a trivial Neumann condition and therefore is constant as claimed \cite[Theorem 4.18]{DahKen1987}. 

Now suppose instead that $\varphi_0\in \Ker \vec\Lambda$. Since $\nabla w_0$ is a monogenic constant with vanishing tangential trace,
\begin{align*}
   \Hi[\varphi_0] =& -\tr\Tiii[\nabla w_0]+\Hi \circ \tr \Ti[\nabla w_0]
\end{align*}
lies in the image of the finite-dimensional space $\SI\tng(\Omega)$ (see Proposition \ref{prop:monogenic constants}). Thus $\dim \Hi(\Ker \vec \Lambda)<\infty$. Applying the left inverse $\G$ of $\Hi$ given in Proposition \ref{prop:inverse-Hilbert} we have the second claim.

Finally, since $\partial\Omega$ is connected, $\SI\tng(\Omega)=0$
because $\SI\tng(\Omega)$ is isomorphic to the second real cohomology
space \cite{Berger1977}, so $\Hi(\Ker \vec \Lambda)=0$. Thus
$\Ker \vec \Lambda \subseteq \Ker \Hi=\R$. Clearly $\vec \Lambda$
annihilates constants, so the proof is finished.
\end{proof}

  
\section{Hilbert transform associated to the main Vekua
  equation \label{sec:f-HT}}

The general Vekua equation $DW=aW+b\overline{W}$, whose theory was
introduced in \cite{Bers1953,Vekua1959} for functions in $\R^2$, plays
an important role in the theory of pseudo-analytic functions, which
has since been been extended to wider contexts, including quaternionic
analysis \cite{Berglez2007,Malonek1998}. The definition of the Hilbert transform $\Hi$
for monogenic functions now permits us to define the analogous Hilbert
transform $\Hi_{f}$ associated to the \textit{main Vekua equation}
\begin{align}\label{eq:Vekua}
   DW=\frac{Df}{f}\overline W.
\end{align}

Following the vocabulary used in \cite{DelPor2017} we will say that
$f^2$ is a \textit{conductivity} when $f$ is a non-vanishing
$\R$-valued function in the domain under consideration. The
conductivity will be called \textit{proper} when
$\rho(f)=\sup(|f|, 1/|f|)$ is finite. Most of the time
$f\in W^{1,\infty}(\Omega,\R)$. Note that $f$ and $(1/f)\vec u$ are
simple examples of solutions of \eqref{eq:Vekua}, where
$\vec u\in\SI(\Omega)$ is a vectorial monogenic constant.  We now extend some of our previous
results, which are applicable to $f\equiv1$, to the more general
equation \eqref{eq:Vekua}.

\subsection{Construction of the Vekua-Hilbert
  transform \label{subsec:construction}}

Results in \cite[Chapter 16]{Krav2009} relate solutions of the main
Vekua equation to solutions of other differential equations. Note that
$W=W_0+\vec W$ satisfies \eqref{eq:Vekua} if and only if the scalar
part $W_0$ and the vector part $\vec W$ satisfy the following
homogeneous div-curl system:
\begin{align}\label{eq:systeK_1} 
  \div(f\vec W) &= 0, \nonumber\\
  \curl(f\vec W) &= -f^2\nabla (W_0/f).
\end{align}
In particular, the vector part $\vec W$ must satisfy the
\textit{double curl-type equation}
\begin{align}  \label{eq:doublerot} 
  \curl (f^{-2} \curl (f\vec W )) = 0   
\end{align}
while the scalar part $W_0$ is a solution of the \textit{conductivity equation} 
\begin{align}  	\label{eq:conductivity} 
  \nabla\cdot f^2\nabla(W_0/f) = 0.
\end{align}
 
The following fact is derived from a basic estimate on elliptic
boundary problems.
\begin{lema} {\rm(\cite[Theorem 4.1]{Isakov1998}, \cite[Theorem
    10]{Mikh1978} see also
    \cite{Evans1998,Gilbarg1983,Grisvard1985})} \label{lema:solution_conductivity}
  Let $\Omega$ be a bounded domain in $\R^3$ with Lipschitz boundary
  and let $f^2\in W^{1,\infty}(\Omega,\R)$ be a proper
  conductivity. Suppose that $\varphi_0\in H^{1/2}(\partial\Omega,\R)$
  is known. Then there exists a unique extension
  $W_0\in W^{1,2}(\Omega,\R)$ satisfying \eqref{eq:conductivity} such
  that
\begin{align}\label{eq:extension_conductivity}
   \tr (W_0/f)&=\varphi_0   
\end{align}
 on   $\partial\Omega$. Further,
\begin{align}\label{eq:minimiza_W_0}
   \|W_0/f\|_{W^{1,2}(\Omega,\R)}\leq C_{\Omega,\rho(f)} \|\varphi_0\|_{H^{1/2}(\partial\Omega,\R)}
\end{align}
where $C_{\Omega,\rho(f)}$ only depends on $\Omega$ and $\rho(f)$.
\end{lema}

In Section \ref{sec:Dirichlet-Neumann} we will define a natural
Neumann data for the conductivity equation \eqref{eq:conductivity}.
We will also prove a version of Lemma \ref{lema:solution_conductivity}
for the vector part $\vec W$ of solutions of the Vekua equation.

To define the Hilbert transform for \eqref{eq:Vekua}, let
$\varphi_0\in H^{1/2}(\partial\Omega,\R)$ be a scalar boundary value
function, and apply Lemma \ref{lema:solution_conductivity} to obtain $W_0$.
The decomposition \eqref{eq:decomposition_T} of the Teodorescu operator
applied to vector fields reduces to
\begin{align*} 
  T_{\Omega}\left[-f^2\nabla(W_0/f)\right] = 
  \Ti\left[-f^2\nabla (W_0/f)\right]+\Tiii\left[-f^2\nabla(W_0/f)\right],
\end{align*}
and both components lie in $W^{1,2}(\Omega)$.   

\begin{definicion}\label{def:f_HT}
  The \textit{Vekua-Hilbert transform}
  \[\Hi_{f}\colon H^{1/2}(\partial\Omega,\R)\to
  H^{1/2}(\partial\Omega,\R^3)\]
  \textit{associated to the main Vekua equation} \eqref{eq:Vekua} is
  given by
\begin{align}\label{eq:Hilbert_general_1-3}
    \Hi_f[\varphi_0] =\vec\alpha-\Hi[\alpha_0],
\end{align}
where $\Hi$ is the Hilbert transform $\Hi$ defined in
\eqref{eq:Hilbert_transform}, the associated Teodorescu traces are
\begin{align}\label{eq:g0gbarra}
  \alpha_0 = \tr\Ti\left[-f^2\nabla(W_0/f) \right],\
  \vec\alpha = \tr\Tiii\left[-f^2\nabla (W_0/f)\right] ,
\end{align} 
and $W_0$ is the solution of the conductivity equation
\eqref{eq:conductivity} satisfying the boundary condition
\eqref{eq:extension_conductivity}.  
\end{definicion}

By the Trace Theorem we have
$\alpha=\alpha_0+\vec\alpha \in H^{1/2}(\partial\Omega,\H)$, and in
fact by \cite[Proposition 3.2]{DelPor2017},
$\vec\alpha\in \Sol(\partial\Omega,{\R}^3)$. Similarly to the Hilbert
transform $\Hi$ for the monogenic case, $\Hi_f$ can be expressed as
\begin{align}\label{eq:Hif}
   \Hi_{f}[\varphi_0]=\vec\alpha - \frac{1}{2}\vecK[h_f]
\end{align}
with the real-valued function  
\begin{align} \label{eq:hf}
   h_f =2(I+K_0)^{-1}\alpha_0 \in H^{1/2}(\partial\Omega,\R).
\end{align}

 The term ``Vekua-Hilbert transform'' is justified by the following.

\begin{teorema}\label{theo:Hf}
  Let $\Omega$ be a bounded Lipschitz domain and let $f\in W^{1,\infty} 
	(\Omega,\R)$ be a proper conductivity. Suppose
  that $\varphi_0 \in H^{1/2}(\partial\Omega,\R)$ . Then the
  quaternionic function
\begin{equation}
   f\varphi_0+(1/f)\Hi_f[\varphi_0]
\end{equation}
is the trace of a solution of the main Vekua equation
\eqref{eq:Vekua}.
\end{teorema}

\begin{proof}
  To produce $W=W_0+\vec W\in W^{1,2}(\Omega,\H)$ satisfying
  \eqref{eq:Vekua} such that
\begin{align*} 
  \tr W_0  =  f\varphi_0  ,\ \tr {f\vec W} =  \Hi_f[\varphi_0] ,
\end{align*}
we take the extension $W_0$ of $f\varphi_0$ given by Lemma
  \ref{lema:solution_conductivity}, and define the vector part
  $\vec W$ by
\begin{align}\label{eq:extension_vectorial}
   f\vec W = T_{\Omega}[\vec v]+
	 F_{\partial\Omega}[\Hi_{f}[\varphi_0]] = \Tiii[\vec v]-\Fii[h_f],
\end{align}
with $\vec v=-f^2\nabla (W_0/f)$ and $h_f$ given by \eqref{eq:hf};
recall also \eqref{eq:decomposition_T}--\eqref{eq:F0_F1}. Since
\eqref{eq:Vekua} is equivalent to $\div(f\vec W)=0$,
$\curl(f\vec W)=\vec v$, i.e.\ a div-curl system
\eqref{eq:div_curl_system} with $g_0=0$, $\vec g=\vec v$, it follows
from Theorem \ref{theo:div-curl}
that $W=W_0+\vec W$ is
a solution of \eqref{eq:Vekua}. Further, from \eqref{eq:nt-limit} we
have that
\begin{align}\label{eq:limite_fw}
    \tr {f\vec W}=\alpha_0+\vec\alpha+\tr F_{\partial\Omega}
    [\vec\alpha-\Hi[\alpha_0]]
    =\vec\alpha - \Hi[\alpha_0]   = \Hi_f[\varphi_0] 
\end{align}
as required.
\end{proof}

\begin{remark}\label{rem:f-1}
  When $f\equiv 1$, the transform $\Hi_f$ coincides with the Hilbert
  transform $\Hi$ of the monogenic case. To see this more clearly, first
  note that by \eqref{eq:conductivity}, $W_0$ must be the harmonic
  extension of $\varphi_0$ to $\Omega$; similarly 
	$\alpha_0=-\tr\Ti[\nabla W_0]$ and $\vec\alpha=-\tr\Tiii[\nabla W_0]$
	.  Now \eqref{eq:HT2T0} says
  that $\Hi[\varphi_0]$ is precisely the definition of $\Hi_{f\equiv 1}
	[\varphi_0]$. 
\end{remark}

\begin{remark}
  In \cite[Remark\ 5.6]{DelPor2017} a slightly different definition was
  proposed for $\Hi_{f}$ in terms of the operators $\Ti,\Tiii$ and a
  certain radial integration operator, used in providing a general
  solution to the div-curl system valid in star-shaped domains. In that
  definition it is not possible to show the relationship with the
  monogenic Hilbert transform, because its construction is completely
  interior to domain $\Omega$.
\end{remark}

\subsection{Properties of $\Hi_f$}

\begin{proposicion}\label{prop:kernel-Hf}
  Let $\varphi_0\in H^{1/2}(\partial\Omega,\R)$. Then
  $\varphi_0\in \Ker\Hi_f$ if and only if the associated Teodorescu traces 
	$\alpha_0, \vec\alpha$ vanish.
\end{proposicion}

\begin{proof}
  The Hodge decomposition \cite[Theorem 8.7]{GuHaSpr2008} gives the
  orthogonal direct sum
  $L^2(\Omega,\H)=(\M(\Omega)\cap L^2(\Omega,\H))\oplus D( W_0^{1,2}
	(\Omega,\H))$, where the subscript in $W_0^{1,2}$ indicates zero
	trace. Thus
\[ 
   \tr f\vec W =0 \iff D(f\vec W)\in D(W_0^{1,2}(\Omega,\H)) \iff
   \alpha_0=0,\ \vec\alpha=0,
\]
  with $f\vec W$ as in \eqref{eq:extension_vectorial} and where the
  last equivalence follows from the result \cite[Proposition
  8.9]{GuHaSpr2008}, which identifies orthogonality to all monogenic
  functions with the vanishing of the trace of the Teodorescu
  operator.  By \eqref{eq:limite_fw} we have the result.
\end{proof}
 
We will say that the vector part $\vec W$ of $W$ is
\textit{normalized} when it satisfies the boundary condition
\begin{align}\label{eq:normalization}
   \tr f\vec W = \Hi_f[\varphi_0]. 
\end{align}

Let $W=W_0+\vec W$ be an arbitrary solution of the main Vekua equation
\eqref{eq:Vekua}, and write $\varphi_0=\tr W_0$,
$\vec \varphi=\tr \vec W$. Consider
\begin{align*}
  \vec W^* = \vec W - \frac{1}{f}F_{\partial\Omega}[f\vec \varphi
   +  \Hi_{f}[\varphi_0]].
\end{align*}
Then by \eqref{eq:formula_BP}, $ f\vec W^*$ has
the form \eqref{eq:extension_vectorial} and hence satisfies the
normalization condition \eqref{eq:normalization}, with $W_0+\vec W^*$
a solution of \eqref{eq:Vekua}.

On the other hand, let $W^1$ and $W^2$ be two solutions of
\eqref{eq:Vekua} with the same scalar part and with normalized
vector parts. If $\varphi^i=\tr W^i$, $i=1,2$, then
\begin{align*}
   fW^1-fW^2=T_{\Omega}[D(f(W^1-W^2))]+F_{\partial\Omega}[f\varphi^1
	 -f\varphi^2]=0,
\end{align*}
since $f(W^1-W^2)$ is monogenic. Therefore $W^1=W^2$; i.e.\ there
is only one normalized vector part for a given scalar part of
a solution of the main Vekua equation.

Some important facts about the solvability and regularity of the
conductivity equation \eqref{eq:conductivity} permit us to prove the
boundedness of $\Hi_f$:
\begin{teorema}\label{theo:H_f_bounded}
  Let $\Omega$ be a bounded Lipschitz domain and let
  $f\in W^{1,\infty} (\Omega,\R)$ be a proper conductivity. Then the
  Vekua-Hilbert transform $\Hi_f \colon H^{1/2}(\partial\Omega,\R)$
  $\to H^{1/2}(\partial\Omega,\R^3)$ is a bounded operator, as are
  also the associated Teodorescu traces
  $\varphi_0 \longmapsto \alpha_0$ and
  $ \varphi_0 \longmapsto \vec \alpha$ from
  $H^{1/2}(\partial\Omega,\R)$ to $H^{1/2}(\partial\Omega,\R)$ and
  $H^{1/2}(\partial\Omega,\R^3)$, respectively.
\end{teorema}

\begin{proof}
  Let $\varphi_0 \in H^{1/2}(\partial\Omega,\R)$.  By Lemma
  \ref{lema:solution_conductivity}, take $W_0\in W^{1,2}(\Omega,\R)$
  satisfying \eqref{eq:conductivity}--\eqref{eq:extension_conductivity}.
  Since both $T_{\Omega} \colon L^2(\Omega)\to W^{1,2}(\Omega)$ and
  $\tr\colon W^{1,2}(\Omega) \to H^{1/2}(\partial\Omega)$ are continuous, 
	  by \eqref{eq:minimiza_W_0} we have
\begin{align}\label{eq:bound_1}
    \|\alpha_0+\vec\alpha\|_{H^{1/2}(\partial\Omega)}
    \nonumber &\leq \|\tr\|\|T_{\Omega}\| \|f^2\nabla(W_0/f)\|_
		{L^2(\Omega)}\\
    \nonumber&\leq \|\tr\|\|T_{\Omega}\| \|f\|^2_{L^{\infty}(\Omega)}
		\|W_0/f\|_{W^{1,2}(\Omega)}\\
		&\leq C_{\Omega,\rho(f)}\|\tr\| \|T_{\Omega}\| \|f^2\|_{L^{\infty}(\Omega)} \|\varphi_0\|_{H^{1/2}(\partial\Omega)}.
\end{align}
From this follows the continuity of $\alpha_0$ and $\vec \alpha$.

By the continuity of the Hilbert transform (Theorem \ref{theo:H_bounded_noncompact}),
\begin{align}\label{eq:bound_2}
  \|\Hi[\alpha_0]\|_{H^{1/2}(\partial\Omega)}
   \leq \|\Hi\| \|\alpha_0\|_{H^{1/2}(\partial\Omega)}.
\end{align}
Using the inequalities \eqref{eq:bound_1}--\eqref{eq:bound_2}, we have that
\begin{align*}
   \nonumber \|\Hi_f[\varphi_0]\|_{H^{1/2}(\partial\Omega)}&=
	 \|\vec\alpha-\Hi[\alpha_0]\|_{H^{1/2}(\partial\Omega)}\\
   \nonumber&\leq \max(1,\|\Hi\|) \|\alpha_0+\vec\alpha\|_{H^{1/2}(
	 \partial\Omega)}\\
   \nonumber&\leq C_{\Omega,\rho(f)} \max(1,\|\Hi\|) \|\tr\| 
	 \|T_{\Omega}\| \|f^2\|_{L^{\infty}(\Omega)} \|\varphi_0\|_
	 {H^{1/2}(\partial\Omega)}.
	 \end{align*} 
Therefore $\Hi_f$ is continuous.
\end{proof}

Analogous to the estimates for the solutions to the conductivity
equation \eqref{eq:minimiza_W_0}, we have

\begin{proposicion}\label{prop:estimation-vectorial}
  Let $\Omega$ be a $ C^{1,1}$ bounded Lipschitz domain and let
  $f\in W^{1,\infty}(\Omega,\R)$ be a proper conductivity. Suppose
  that $\varphi_0\in H^{1/2}(\partial\Omega,\R)$. Then the vector
  extension given by \eqref{eq:extension_vectorial} satisfies
\begin{align}\label{eq:extension-bounded}
   \|f\vec W\|_{W^{1,2}(\Omega)}\leq C_{\Omega,\rho(f)}^{*}
	 \|\varphi_0\|_{H^{1/2}(\partial\Omega)} 
\end{align}
where $C_{\Omega,\rho(f)}^{*}$ depends only on $\Omega$ and $\rho(f)$.
\end{proposicion}

\begin{proof}
  Let $W_0\in W^{1,2}(\Omega,\R)$ be the unique solution of
  \eqref{eq:conductivity}--\eqref{eq:extension_conductivity}. Then as in
  \eqref{eq:bound_1},
\begin{align}\label{eq:bound1}
   \|T_{\Omega}[-f^2\nabla (W_0/f)]\|_{W^{1,2}(\Omega)}\leq C_{\Omega,
	 \rho(f)} \|T_{\Omega}\|\|f^2\|_{W^{1,\infty}(\Omega)}
         \|\varphi_0\|_{H^{1/2}(\partial\Omega)}.
\end{align}
By \eqref{eq:F0_F1-div-curl} together with the fact that 
$\|\curl \vec u\|_{L^2}^2+\|\div \vec u\|_{L^2}^2\leq
3\sum_{i=1}^3{\|\nabla(u_i)\|_{L^2}^2}$
for every $\vec u\in W^{1,2}(\Omega)$ and Theorem \ref{theo:T_acotado}(c),
\begin{align}\label{eq:bound2}
   \nonumber\|F_{\partial\Omega}[h_f]\|_{L^2(\Omega)}&=
   \|\div M[\eta h_f]-\curl M[\eta h_f]\|_{L^2(\Omega)}\\
   \nonumber&\leq \sqrt{3}\|\nabla M[\eta h_f]\|_{L^2(\Omega)}\\
   \nonumber&\leq \sqrt{3}\|M[\eta h_f]\|_{W^{1,2}(\Omega)}\\
   &\leq \sqrt{3} C_1 \|M\| \|(I+K_0)^{-1}\|
	 \|\eta\|_{W^{1,\infty}(\Omega)} \|\varphi_0\|_{{H^{1/2}(
	 \partial\Omega)}},
\end{align}
where the constant $C_1$ in the last inequality comes from the fact
that both
$(I+K_0)^{-1}\colon H^{1/2}(\partial\Omega)\to
H^{1/2}(\partial\Omega)$
and $\varphi_0\mapsto \alpha_0$ are bounded (see Proposition
\ref{prop:invertibility-I+K_0}(c) and Theorem
\ref{theo:H_f_bounded}). By \eqref{eq:bound1} and \eqref{eq:bound2}
and by the fact that
$f\vec W=T_{\Omega}[-f^2\nabla (W_0/f)]-F_{\partial\Omega}[h_f]$ from
\eqref{eq:extension_vectorial},
\[
   \|f\vec W\|_{L^2(\Omega)}\leq C_2
	 \|\varphi_0\|_{H^{1/2}(\partial\Omega)},
\] 
where
$C_2=C_{\Omega,\rho(f)} \|T_{\Omega}\|\|f^2\|_{W^{1,\infty}(\Omega)} +
\sqrt{3} C_1 \|M\| \|(I+K_0)^{-1}\| \|\eta\|_{W^{1,\infty}(\Omega)}$.

By the first Friedrichs inequality provided in  Proposition
\ref{prop:Friedrichs-inequalities}, using the div-curl system
\eqref{eq:systeK_1} and the boundedness of the Vekua-Hilbert transform
$\Hi_f$, we have
\begin{align*}
   \|f\vec W\|_{W^{1,2}(\Omega)}^2 &\leq \|f\vec W\|_{L^2(\Omega)}^2+     
	 \|\curl (f\vec W)\|_{L^2(\Omega)}^2 + \|\Hi_f[\varphi_0]\cdot \eta\|_
	 {H^{1/2}(\partial\Omega)}^2\\
   &\leq C_2 \|\varphi_0\|_{H^{1/2}(\partial\Omega)}
	 ^2+\|f^2\nabla(W_0/f)\|_{L^2(\Omega)}^2 \\
	 & \qquad + \|\Hi_f\|^2 \|\eta\|_{W^{1,\infty}(\Omega)}^2 \|\varphi_0\|_
	 {H^{1/2}(\partial\Omega)}^2\\
   &\leq C^{*2}_{\Omega,\rho(f)} \|\varphi_0\|_{H^{1/2}(\partial\Omega)}^2,
\end{align*}
where
$C^{*2}_{\Omega,\rho(f)}=C_2^2+C_{\Omega,\rho(f)}^2\|f^2\|_{L^{\infty(\Omega)}}^2+\|\Hi_f\|^2
\|\eta\|_{W^{1,\infty}(\Omega)}^2$.
\end{proof}

\section{Dirichlet-to-Neumann map for the conductivity
  equation\label{sec:Dirichlet-Neumann}}

The \textit{Dirichlet-to-Neumann map} for the conductivity equation
is \begin{align}\label{eq:D-N}
   \nonumber \Lambda_{0,f^2}\colon H^{1/2}(\partial\Omega,\R)&
   \to H^{-1/2}(\partial\Omega,\R),\\
   \varphi_0 &\mapsto -f^2 \nabla(W_0/f)|_{\partial\Omega} \cdot\eta .
\end{align}
Here $\eta$ is again the unit outer normal vector to $\partial\Omega$
and $W_0\in W^{1,2}(\Omega,\R)$ is the unique extension of
$f\varphi_0$ as a solution of the conductivity equation
\eqref{eq:conductivity} given in Lemma
\ref{lema:solution_conductivity}.  For $f^2$ smooth,
$\Lambda_{0,f^2}[\varphi_0]$ is well-defined pointwise, but for
general proper conductivities, the D-N map is only weakly defined by
the relation
\begin{align}
   \langle \Lambda_{0,f^2} [\varphi_0], \tr v_0 \rangle_{\partial\Omega} 
   = -\int_{\Omega}{f^2 \nabla(W_0/f)\cdot \nabla v_0 \,d\vec y},
\end{align}
where $\nabla\cdot f^2 \nabla (W_0/f)=0$, $\tr W_0=f\varphi_0$ and
$v_0\in W^{1,2}(\Omega,\R)$. One reference for the scalar D-N map is
\cite{Salo2008}. This map is an essential part of the solution of the
Calder\'{o}n problem \cite{Calderon1980}, that is, to recover the
pointwise conductivity $f^2$ interior to the domain $\Omega$ from
electrical current measurements on the boundary $\partial\Omega$.

We will follow the definition given in \cite{SU1987}, but we write
$\Lambda_{0,f^2}$ rather than $\Lambda_{f^2}$ to emphasize the scalar
nature of this quantity.

\subsection{Quaternionic Dirichlet-to-Neumann map\label{subsec:quatDNmap}}
In analogy to \eqref{eq:quaternionic-D-N-monogenic} we introduce  the
following.
\begin{definicion} The \textit{quaternionic Dirichlet-to-Neumann map}
for the conductivity equation is defined strongly by 
\begin{align}\label{operator-D-N-quaternionic}
   \Lambda_{f^2}[\varphi_0] = (f^2D(W_0/f)|_{\partial\Omega})\eta
\end{align}
for functions $\varphi_0$ whose interior extension $W_0$ satisfy the
conductivity equation \eqref{eq:extension_conductivity}.
\end{definicion}

\begin{teorema} \label{theo:weakDNVekua} The weak definition of
$\Lambda_{f^2}\colon H^{1/2}(\partial\Omega,\R) \to
   H^{-1/2}(\partial\Omega,\H)$
 is given by
\begin{align}\label{weak-quaternionic-D-N}
    \nonumber \langle \Lambda_{f^2} [\varphi_0], \tr v \rangle_{\partial 
		\Omega}&=\Sc \int_{\Omega}{f^2 \nabla(W_0/f)
		f(\Dr+(Df/f))[v_0/f] \, d\vec y}\\
	  &+\Sc \int_{\Omega}{f^2\nabla(W_0/f) \frac{1}{f}
		(\Dr-M^{Df/f}C_{\H})[f\vec v]\, d\vec y},		
\end{align}
for every $v\in W^{1,2}(\Omega,\H)$, where $W_0$ is the solution of
\eqref{eq:conductivity} with boundary values
\eqref{eq:extension_conductivity}, $\Dr W=WD$, and $M^{(\cdot)}$
denotes quaternionic multiplication from the right.
\end{teorema}

With the notation $\Lambda_{f^2}=\Lambda_{0,f^2}+\vec\Lambda_{f^2}$
we can express \eqref{weak-quaternionic-D-N} as
\[ \langle \Lambda_{f^2} [\varphi_0], \tr v \rangle_{\partial\Omega} 
 = \langle \Lambda_{0,f^2}[\varphi_0], \tr v_0 \rangle_{
	  \partial\Omega}+\langle \vec \Lambda_{f^2}[\varphi_0], \tr \vec v 
		\rangle_{\partial\Omega},
\]
where the scalar part is indeed the D-N map $\Lambda_{0,f^2}$ of
\eqref{eq:D-N}, and the vector part (or \textit{tangential D-N map})
is
\begin{equation} \label{eq:DNtangencial}
\vec\Lambda_{f^2}[\varphi_0]=
 f^2\nabla(W_0/f)|_{\partial\Omega}\times\eta.
\end{equation}
The operator $\Dr-M^{Df/f}C_{\H}$ appearing in the second integral has
been called \cite{KravTremb2011} the ``Bers derivative'' of solutions
of \eqref{eq:Vekua}. When $f\equiv 1$, \eqref{weak-quaternionic-D-N}
reduces to \eqref{weak-monogenic-D-N}.  The proof of Theorem
\ref{theo:weakDNVekua} is an exercise in vector calculus, based on the
ideas of \eqref{weak-monogenic-D-N}, and the observation that by
Green's formula,
\begin{align*}
   \langle \vec\Lambda_{f^2}[\varphi_0], \tr\vec v \rangle_{\partial\Omega}
   = \int_{\Omega}f^2 \nabla(W_0/f) \cdot \curl\vec v \, d\vec y -
   \int_{\Omega} \curl(f^2 \nabla(W_0/f)) \cdot \vec v \, d\vec y.
\end{align*}  
 
\begin{proposicion}\label{prop:DN-continuous}
  Let $\Omega$ be a bounded domain with Lipschitz boundary, and let
  $f^2\in W^{1,\infty}( \Omega,\R)$ be a proper conductivity. The D-N
  map $\Lambda_{f^2}$ is continuous from $H^{1/2}(\partial\Omega,\R)$
  to the dual space $H^{-1/2}(\partial\Omega,\H)$ of
  $H^{1/2}(\partial\Omega,\H)$.
\end{proposicion}

\begin{proof}
  Let $\varphi_0\in H^{1/2}(\partial\Omega,\R)$. Since
  $\Lambda_{0,f^2}[\varphi_0]$ and $\vec\Lambda_{f^2}[\varphi_0]$ are
  the normal and tangential traces respectively of
  $f^2\nabla(W_0/f)\in W^{2,\div}(\Omega,\R^3)\cap W^{2,\curl}
  (\Omega,\R^3)$, by \eqref{eq:minimiza_W_0},
\begin{align}\label{bound_extension}
   \|\nabla(W_0/f)\|_{L^2} \le  \|W_0/f\|_{W^{1,2}}
   \le C_{\Omega,\rho(f)} \|\varphi_0\|_{H^{1/2}}.
\end{align}
By \cite[Theorems 1, 2]{Dautray1985}, we have that $\gamma\nrm$ and $\gamma\tng$ are bounded linear mappings from $W^{2,\div}(\Omega,\R^3)$ to $\H^{-1/2}(\Omega,\R)$ and from $W^{2,\curl}(\Omega,\R^3)$ to $\H^{-1/2}(\Omega,\R^3)$, so
\begin{align*}
   \|\Lambda_{0,f^2}[\varphi_0]\|_{H^{-1/2}}&\leq \|\gamma\nrm\| \|f^2
	 \nabla(W_0/f)
	 \|_{L^2},\\
   \|\vec \Lambda_{f^2}[\varphi_0]\|_{H^{-1/2}}&\leq \|\gamma\tng\| \big(
	 \|f^2\nabla(W_0/f)\|_{L^2}+\|\curl(f^2\nabla(W_0/f))\|_{L^2} \big).
\end{align*}
Since 
\begin{align}\label{eq:inequality-curl}
   \|\curl(f^2\nabla(W_0/f))\|_{L^2}\leq \|\nabla f^2\|_{L^{\infty}}
	 \|\nabla(W_0/f)\|_{L^2}, 
\end{align}
by \eqref{bound_extension} we have
\begin{align}\label{bound-curl}
   \|\vec \Lambda_{f^2}[\varphi_0]\|_{H^{-1/2}}\leq C_{\Omega,\rho(f)}
	\|\gamma\tng\|\|f^2\|_{W^{1,\infty}}\|\varphi_0\|_{H^{1/2}}.
\end{align}
Then
\[  \|\Lambda_{f^2}[\varphi_0]\|_{H^{-1/2}(\partial\Omega,\H)}\leq C_3
	 \|\varphi_0\|_{H^{1/2}(\partial\Omega,\R)}   
\]
where $C_3=C_{\Omega,\rho(f)}(\|\gamma\nrm\| \|f\|_{L^{\infty}}+\|\gamma\tng\| \|f\|_{W^{1,\infty}})$.
\end{proof}

Proposition \ref{prop:DN-continuous} justifies the claim made for the
codomain of the D-N map for the monogenic case given in
\eqref{eq:quaternionic-D-N-monogenic}.

In the context of $\R^2$, the classical D-N map coincides with
the tangential derivative of the Hilbert transform \cite[Proposition
4.1]{AP2006-1}. In $\R^3$ the situation is intrinsically more
complicated; some relations between the operators $\Lambda_{0,f^2}$,
$\vec \Lambda_{f^2}$, and $\Hi_f$ will be developed in subsection
\ref{subsec:boundedness}. Here we only note that $\Lambda_{f^2}$
can be rewritten in various ways, as a consequence of
$\tr f\vec W=\Hi_f[\varphi_0]$ (Theorem \ref{theo:Hf}):
\begin{align}\label{eq:relation_operators}
    \nonumber  \Lambda_{f^2}[\varphi_0] &= (f^2\nabla(W_0/f)|_{
		\partial\Omega})\eta  
		\nonumber =\curl f\vec W\Big|_{\partial\Omega}   \cdot \eta - \curl 
		f\vec W\Big|_{\partial\Omega} \times   \eta\\
		&= D(f\vec W)\Big|_{\partial\Omega} \cdot \eta - D(f\vec W)
		\Big|_{\partial\Omega} \times \eta=-(D(f\vec W)|_{\partial\Omega})
		\eta		
\end{align}   
where $W_0+\vec W$ is a solution of the main Vekua equation.
  
\subsection{Norm properties of $\Hi_f$ \label{subsec:boundedness}}

Since the Vekua-Hilbert transform $\Hi_f$ is a generalization of the
Hilbert transform $\Hi$, it is natural that $\Hi_f$ preserves many of
its properties; we will make use of the D-N mapping to investigate
them.  First we relate the Vekua-Hilbert transform $\Hi_f$ to the
scalar D-N map $\Lambda_{0,f^2}$ and the vectorial D-N map
$\vec\Lambda_{f^2}$ through the operator compositions
\eqref{eq:composition-H} and \eqref{eq:composition-Hf}.

\begin{proposicion}\label{prop:new-Hf}
  Let $\Omega\subseteq\R^3$ be a bounded Lipschitz domain and let
  $f^2\in W^{1,\infty}(\Omega,\R)$ be a proper conductivity. Then
  $\Hi_f$ can be written as
\begin{align}\label{Hf-def-2}
  \Hi_f[\varphi_0] = \tr M \vec \Lambda_{f^2}[\varphi_0] -
  \Hi \tr M \Lambda_{0,f^2}[\varphi_0] + \tr L[\curl(f^2\nabla(W_0/f))].
\end{align}
\end{proposicion} 

\begin{proof} 
  First suppose that in fact
  $\varphi_0\in H^{3/2}(\partial\Omega,\R)$.  Take
  $W_0\in W^{1,2}(\Omega,\R)$ satisfying
  \eqref{eq:conductivity}--\eqref{eq:extension_conductivity}. Since
  $\nabla(W_0/f)\in W^{1,2}(\Omega,\R^3)$, by the proof of Proposition
  \ref{prop:T=-MLambdatr} we have that
\begin{align}\label{eq:rewritten-T0}  
   \Ti[f^2\nabla(W_0/f)]= - M[\Lambda_{0,f^2}[\varphi_0]].
\end{align}
Thus we consider the associated trace $\alpha_0$ in the Vekua-Hilbert
transform $\Hi_f$ \eqref{eq:Hilbert_general_1-3} as constructed in the
following way:
  \begin{align*}
  \xymatrix{
      H^{1/2}(\partial\Omega,\R)  \ar[r] ^{\Lambda_{0,f^2}} & 
      H^{-1/2}(\partial\Omega,\R)  \ar[r]^{\tr M} &
      H^{1/2}(\partial\Omega,\R),
  }
\end{align*}
i.e.
\begin{align}\label{eq:composition-H}
   \Hi[\alpha_0] = \Hi \tr M  \Lambda_{0,f^2}[\varphi_0].
\end{align}

In the proof of \cite[Proposition 4.8]{DelPor2017} the relation
$\Tiii[\vec w]=-M[\vec w|_{\partial\Omega}\times\eta]-L[\curl\vec w]$
was proved for all $\vec w\in W^{1,2}(\Omega,\R^3)$, where $L$ is the
right inverse of the Laplacian $\Delta$ given by
\begin{align*}
   L[w](\vec x)=\int_{\Omega}{\frac{w(\vec y)}{4\pi|\vec x-\vec y|} \, d
	 \vec y}.
\end{align*}
Therefore by  \eqref{eq:DNtangencial}
\begin{align}\label{eq:composition-Hf}
	   \vec \alpha= \tr M \vec \Lambda_{f^2}
		 [\varphi_0]+\tr L[\curl (f^2\nabla(W_0/f))].
\end{align}

Thus \eqref{eq:composition-H}--\eqref{eq:composition-Hf} produce the
expression \eqref{Hf-def-2} claimed for $\Hi_f$. This representation
for $\Hi_f$ has been proved for functions $\varphi_0$ in the dense
subspace $H^{3/2}(\partial\Omega,\R)$, and by continuity is valid in
the full space $H^{1/2}(\partial\Omega,\R)$.
\end{proof}

\begin{proposicion} \label{prop:kerHiLambda}
$\R \subseteq \Ker \Hi_f\cap H^{3/2}(\partial\Omega,\R)\subseteq\Ker \Lambda_{0,f^2}\subseteq H^{3/2}(\partial\Omega,\R)$.
\end{proposicion}

\begin{proof}
  The first containment is straightforward from the uniqueness of the
  solutions of the conductivity equation.  The proof of the second
  containment is a consequence of Proposition \ref{prop:kernel-Hf},
  equation \eqref{eq:composition-H} and the fact that $\tr M$ is an
  invertible operator from $L^2(\partial\Omega)$ to
  $H^{1/2}(\partial\Omega)$ \cite[Theorem 3.3]{Ver1984}. Finally, the 
	third containment follows from Proposition \ref{prop:div-curl-spaces}(a).
\end{proof}

At the end of this section we will show that $\Ker\Hi_f$ in fact
consists only of constants. We do not know whether the second
containment of Proposition \ref{prop:kerHiLambda} is an equality
for nonconstant $f$.
 
In \cite[Theorem 0.2]{SU1988} some estimates were presented to establish
the continuous dependence of the scalar D-N map $\Lambda_{0,f^2}$ on
the boundary values of the conductivity $f^2$. More specifically, as
consequence of \cite[Theorem 3.5]{SU1988} we have that
$\|\Lambda_{0,f_n^2}-\Lambda_{0,f^2}\|\longrightarrow 0$ when
$f_n\rightarrow f$ in $L^{\infty}$. We give a similar
result for the quaternionic D-N map $\Lambda_{f^2}$ and for the
Vekua-Hilbert transform $\Hi_f$:

\begin{teorema}\label{Th-dependence}
  Let $\Omega$ be a bounded Lipschitz domain. Let
  $\{f_n\}\subseteq W^{1,\infty} (\Omega,\R)$ be a sequence of proper
  conductivities. Then
  \begin{enumerate}
  \item[(a)] if $f_n\to f$ in $L^{\infty}(\Omega,\R)$,
     then $\|\Hi_{f_n}-\Hi_f\| \longrightarrow 0$;
  \item[(b)] if $f_n\to f$ in $W^{1,\infty}(\Omega,\R)$, 
     then $\|\Lambda_{f_n^2}-\Lambda_{f^2}\| \longrightarrow 0$;
  \end{enumerate}	
as operators on $H^{1/2}(\partial\Omega,\R)$.
\end{teorema}

\begin{proof}
  Let $\varphi_0 \in H^{1/2}(\partial\Omega,\R)$. Let
  $W_{0,n}, W_0\in W^{1,2}(\Omega,\R)$ be the respective extensions to
  solutions of the conductivity equations; that is,
\begin{align*}
   \nabla \cdot f_n^2\,\nabla \left(W_{0,n}/f_n\right) &=0,
    \quad \tr (W_{0,n}/f_n)=\varphi_0,\\
   \nabla \cdot f^2\,\nabla \left(W_0/f\right) &=0,   
 \quad \tr (W_{0}/f)=\varphi_0.
\end{align*}
By \eqref{eq:minimiza_W_0}, these unique solutions satisfy
\[
\left\|\nabla(W_{0,n}/f_n)\right\|_{L^2}\leq
C_{\Omega,\rho(f_n)} \|\varphi_0\|_{H^{1/2}},\quad \left\|\nabla(W_0/f)\right\|_{L^2}\leq C_{\Omega,\rho(f)}
\|\varphi_0\|_{H^{1/2}}.  
\]
It is a well-known fact about elliptic equations \cite[Proposition 3.3]{SU1988} that 
\begin{align}\label{eq:bound-nabla}
  \|\nabla(W_{0,n}/f_n-W_{0}/f)\|_{L^2(\Omega)}\leq c_n 
	\|\varphi_0\|_{H^{1/2}(\partial\Omega)},
\end{align}	
where 
\[
  c_n=\frac{\sup |f_n^2-f^2|}{\inf f_n^2}\left(1+\left(\frac{\sup f^2}{ 
	\inf f^2}\right)^{1/2}\right).
\]	
Now consider the traces
\begin{align*}
   \alpha_n &= \alpha_{0,n}+\vec\alpha_n=\tr T_{\Omega}[-f_n^2\nabla
	 (W_{0,n}/f_n)], \\
   \alpha &= \alpha_0+\vec\alpha=\tr T_{\Omega}[-f^2\nabla(W_0/f)]. 
\end{align*}
By \eqref{eq:bound-nabla}, we have
\begin{align}\label{bound-f-nabla} 
  \|f_n^2 & \nabla(W_{0,n}/f_n) - f^2\nabla(W_0/f)\|_{L^2}   \nonumber \\
   & \leq \|f_n^2-f^2\|_{L^{\infty}}\|\nabla(W_{0,n}/f_n )\|_{L^2}
     +\|f^2\|_{L^{\infty}} \|\nabla(W_{0,n}/f_n-W_0/f) \|_{L^2}  \nonumber  \\
   & \leq \big(C_{\Omega,\rho(f_n)}\|f_n^2-f^2\|_{L^{\infty}}
	 +c_n\|f^2\|_{L^{\infty}}\big) \|\varphi_0\|_{H^{1/2}}.
\end{align}
By \eqref{bound-f-nabla} and the boundedness of the operators $\tr$
and $T_{\Omega}$ we have that
\begin{align*}
   \|\vec \alpha_n - \vec \alpha\|_{H^{1/2}} &\leq 
	 \|\tr\| \|\Tiii\| \|f_n^2\nabla(W_{0,n}/f_n) - f^2\nabla(W_0/f)\|_
	{L^2}\\
  &\leq \|\tr\| \|\Tiii\|\big(C_{\Omega,\rho(f_n)}\|f_n^2-f^2\|_
	{L^{\infty}} + c_n\|f^2\|_{L^{\infty}}\big)\|\varphi_0\|_{H^{1/2}}.
\end{align*}
Analogously,
\begin{align*}
   \|\Hi[&\alpha_{0,n}]-\Hi[\alpha_0]\|_{H^{1/2}}   \\ 
    &\leq \|\Hi\| \|\tr\| \|\Ti\| \big(C_{\Omega,\rho(f_n)}
   \|f_n^2-f^2\|_{L^{\infty}} + c_n\|f^2\|_{L^{\infty}}\big) \|\varphi_0\|_{H^{1/2}}.	
\end{align*}
Since $c_n\rightarrow 0$, we obtain the limit of part (a). For  
part (b), by \eqref{bound-curl} and \eqref{bound-f-nabla} we have
\begin{align*}
   \|\vec\Lambda_{f_n^2}& [\varphi_0]-\vec \Lambda_{f^2}[\varphi_0]
   \|_{H^{-1/2}} \\
   &\leq \|\gamma\tng\| \big(
	 \|f_n^2\nabla(W_{0,n}/f_n) - f^2\nabla(W_0/f)\|_{L^2} \\
   & \qquad\quad  +\|\curl(f_n^2\nabla(W_{0,n}/f_n) - f^2\nabla(W_0/f))\|_{L^2} \big)\\
   &\leq\|\gamma\tng\|\big( C_{\Omega,\rho(f_n)} \|f_n^2-f^2\|_{W^{1,\infty}} + c_n\|f^2\|_{W^{1,\infty}}\big)
    \|\varphi_0\|_{H^{1/2}},
\end{align*}
as required.  \qedhere
\end{proof}

The stability question of the scalar D-N map asks whether two
conductivities $f_1^2$, $f_2^2$ are close whenever $\Lambda_{0,f_1^2}$
is close to $\Lambda_{0,f_2^2}$.  In \cite[Theorem 1]{Alessandrini1988}
it was proved that there exists a continuous nondecreasing function
$\omega \colon [0,\infty) \to [0,\infty)$ satisfying $\omega(t)\to 0$
as $t \to 0^+$ such that
\begin{align*}
   \|f_1^2-f_2^2\|_{L^{\infty}(\Omega)} \leq \omega (\|\Lambda_{0,
	 f_1^2}-\Lambda_{0,f_2^2}\|),
\end{align*}
for a bounded open set $\Omega\subseteq \R^n$ with smooth boundary,
$f_i \in W^{s,2}(\Omega,\R)$ ($L^2$ functions with derivatives up to
order $s$ in $L^2$), $s>n/2$ and $n\geq3$. However, the stability of
the vector part $\vec\Lambda_{f^2}$ remains an open question.

In Theorem \ref{theo:H_f_bounded} it was established that
$\varphi_0 \longmapsto \alpha_0$ and
$\varphi_0\longmapsto \vec \alpha$ are continuous. We will prove that
these mappings are in fact compact when restricted to
$\Ker\Lambda_{0,f^2}$ or $\Ker\vec \Lambda_{f^2}$. 

\begin{proposicion}
  Let $\Omega$ be a bounded  $C^{1,1}$ Lipschitz domain and let
  $f\in W^{1,\infty}(\Omega,\R)$ be a proper conductivity. The
  restrictions of $\Hi_f$ to $\Ker\Lambda_{0,f^2}$ and to
  $\Ker \vec \Lambda_{f^2}$ are compact mappings into
  $H^{1/2}(\partial\Omega,\R^3)$.
\end{proposicion}

\begin{proof}
  Let $\varphi_0\in \Ker \Lambda_{0,f^2}[\varphi_0]$, so the
  associated Teodorescu traces $\alpha_0,\vec \alpha$ are constructed
  as
\begin{align}\label{eq:mappings}
   \begin{array}{lllll}
   \Ker \Lambda_{0,f^2} &\longrightarrow W\nrm^{2,\div\text{-}\curl}(\Omega,\R^3)
 &\hookrightarrow L^2(\Omega,\R^3) 
 &\longrightarrow H^{1/2}(\partial\Omega,\H),\\ 
 \varphi_0 &\longmapsto g &\longmapsto g 
 &\longmapsto \alpha_0+\vec \alpha,
  \end{array}
\end{align}
where $g=f^2\nabla(W_0/f)$ and $\alpha_0+\vec \alpha=-\tr
T_{\Omega}[g]$.
By \eqref{eq:minimiza_W_0},
\eqref{bound_extension}--\eqref{eq:inequality-curl}, the first mapping
of \eqref{eq:mappings} $\varphi_0\mapsto g$ is a
bounded operator from $H^{1/2}(\partial\Omega,\R)$ to
$W\nrm^{2,\div\text{-}\curl}(\Omega,\R^3)$. Thus in fact all of the
mappings shown are bounded. By Proposition
\ref{prop:div-curl-spaces}, the inclusion mapping of
$W\nrm^{2,\div\text{-}\curl}(\Omega,\R^3)$ into $L^2(\Omega,\R^3)$ is
compact. Therefore  $\varphi_0 \mapsto \alpha_0, \vec \alpha$ are compact,
and in consequence $\Hi_f$ is also compact on $\Ker\Lambda_{0,f^2}$ as
claimed. The proof for $\Ker \vec \Lambda_{f^2}$ is similar.
\end{proof}

In the following result we describe the Vekua-Hilbert transform
$\Hi_f$ restricted to the kernel of the D-N operator
$\Lambda_{0,f^2}$.

\begin{teorema} \label{th:KerLambda0}
  Let $\Omega$ be a $C^{1,1}$ bounded Lipschitz domain and let
  $f\in W^{1,\infty}(\Omega,\R)$ be a proper conductivity. Then 
	the Vekua-Hilbert transform $\Hi_f$ restricted to 
	$\Ker \Lambda_{0,f^2}$ produces boundary values of  monogenic 
	constants in $\Omega^-$ which vanish at $\infty$. 
\end{teorema}

\begin{proof}
  Let $\varphi_0 \in \Ker \Lambda_{0,f^2}$. By Proposition 
	\ref{prop:kerHiLambda}, $f^2\nabla(W_0/f)\in W^{1,2}(\Omega,\R^3)$
	. Taking the trace of 
	\eqref{eq:rewritten-T0} we have 
	$\alpha_0=-\tr\Ti[f^2\nabla(W_0/f)]=0$, so
  \begin{align}\label{eq:Hf-reduced}
    \Hi_f[\varphi_0]&=\vec \alpha=\tr \Tiii[-f^2\nabla(W_0/f)].
  \end{align}
  It is a classical fact \cite[Proposition 8.1]{GuHaSpr2008} that
  $T_{\Omega}[w] (\vec x)$ is always monogenic in $\Omega^-$ and tends
  to zero for $|\vec x|\rightarrow\infty$, so $\Ti[f^2\nabla(W_0/f)]$
  vanishes in $\Omega^-$. By the conductivity equation,
  $f^2\nabla (W_0/f)$ is solenoidal and therefore \cite[Proposition
  3.1(i)]{DelPor2017} says that $\Ti[f^2\nabla(W_0/f)]$ is harmonic in
  $\Omega$, hence it vanishes in all of $\R^3$.  Thus the vector field
  $\Tiii[f^2\nabla(W_0/f)]=T_{\Omega}[f^2\nabla (W_0/f)]$ is a
  monogenic constant in $\Omega^{-}$ vanishing at $\infty$, and
  the assertion follows from \eqref{eq:Hf-reduced}.
\end{proof}

By Theorem \ref{th:KerLambda0} and Proposition \ref{prop:monogenic
  constants}, we know that
\begin{align*}
    \dim (\Ker \Hi_f|_{\Ker \Lambda_{0,f^2}})\leq \dim
    \SI\tng(\Omega^-)<\infty.
\end{align*}
By Proposition \ref{prop:kerHiLambda}, we have
$\dim \Ker \Hi_f<\infty$. Therefore, since
$\partial\Omega^-=\partial\Omega$ is connected we have
$\Ker \Hi_f=\R$.  We possess little information about the nature of
$\Ker\Lambda_{0,f^2}$.  It would be interesting, for example, to know
whether all boundary values of exterior monogenic constants vanishing
at $\infty$ are as Theorem \ref{th:KerLambda0}.

\appendix
\section{Div-curl system and  generalizations\label{sec:div-curl-section}}
 
A constructive solution for the div-curl system was presented in
\cite[Theorem 4.4]{DelPor2017} for star-shaped domains in
$\Omega\subseteq\R^3$. We continue to assume that
$\R^3\setminus \Omega$ is connected. Let $g_0\in L^p(\Omega,\R)$ and
$\vec g\in L^p(\Omega, \R^3)$. The div-curl system is
\begin{align}\label{eq:div_curl_system}
   \div\vec w = g_0, \quad  \curl\vec w = \vec g.
\end{align}
Note that $\vec g$ is required to be weakly solenoidal,
\begin{align*}
    \int_{\Omega}{\vec g \cdot \nabla v_0 \, d\vec x}=0  
\end{align*}
for all test functions $v_0\in W_0^{1,q}(\Omega,\R)$, in order
for there to exist solutions to the second equation.
Since $T_{\Omega}[\vec g]\in W^{1,p}(\Omega,\H)$, the scalar function and
the vector field
\begin{align}\label{eq:boundary-Ti}
   \alpha_0=\tr \Ti[\vec g],\  \vec \alpha=\tr \Tiii[\vec g],
\end{align}
are well-defined and
$\alpha=\alpha_0+\vec \alpha=\tr T_{\Omega}[\vec g] \in
W^{1-1/p,p}(\partial\Omega,\H)$.

We now remove the restriction of starshapedness, presenting a solution
of \eqref{eq:div_curl_system} for bounded Lipschitz domains with
weaker topological constraints (for example, a solid torus will be
admissible).  This more general div-curl solution is expressed in
terms of the operators conforming the Hilbert transform $\Hi$
\eqref{eq:Hilbert_transform}, as well as the Teodorescu transform
$T_{\Omega}$ \eqref{eq:T1_T2_T3} and the Cauchy operator
$F_{\partial\Omega}$ \eqref{eq:F0_F1}.

The hypothesis on $\partial\Omega$ is to guarantee that the operator
$I+K_0$ is invertible in $L^p(\partial\Omega,\R)$; it uses the value
of $\epsilon(\Omega)$ which depends only of the Lipschitz character of
$\partial\Omega$, as discussed in Proposition
\ref{prop:invertibility-I+K_0}.

\begin{teorema}\label{theo:div-curl} Let $\Omega$ be a bounded
$C^{1,\gamma}$ Lipschitz domain with $\gamma>0$ and $1<p<\infty$, or a
bounded Lipschitz domain for $2-\epsilon(\Omega)<p<\infty$. Then a
weak solution $\vec w$ of the div-curl system
\eqref{eq:div_curl_system} is given by
\begin{align}\label{eq:sol-div-curl} \nonumber \vec
 w &= T_{\Omega}[-g_0+\vec g] + F_{\partial\Omega}[\vec
\alpha-\Hi(\alpha_0)]\\
 &= -\Tii[g_0] + \Tiii[\vec g] - \Fii[2(I+K_0)^{-1}\alpha_0],		
\end{align}	
where $\alpha_0$ and $\vec \alpha$ were defined in
(\ref{eq:boundary-Ti}). This solution is unique up to adding an
arbitrary monogenic constant. Moreover, $\vec w\in W^{1,p}(\Omega, \R^3)$ when $1<p<2+\epsilon(\Omega)$.
\end{teorema}

\begin{proof}
  The function $h_0=2(I+K_0)^{-1}\alpha_0$ lies in
  $L^p(\partial\Omega,\R)$, and  when
  $1<p<2+\epsilon(\Omega)$, in fact
  $h_0\in W^{1-1/p,p}(\partial\Omega,\R)$ by Proposition
  \ref{prop:invertibility-I+K_0}(b). By \eqref{eq:u+Hu},
\begin{align*}
  \tr_+ F_{\partial\Omega}[h_0]=\alpha_0 + \Hi[\alpha_0].
\end{align*}
Therefore $\Ti[\vec g] + \Fii[h_0]$ is monogenic.
Following the same argument used in the proof of Theorem 4.4 of
\cite{DelPor2017}, we see that
\begin{align*}
   D\vec w=DT_{\Omega}[-g_0+\vec g]-D(\Ti[\vec g] + \Fii[h_0])=-g_0 + 
	 \vec g.
\end{align*}
The function $\vec w$ is purely vectorial because
\begin{align*}
    F_{\partial\Omega}[\vec \alpha-\Hi[\alpha_0]] &= -F_{\partial\Omega}
		[h_0],\\
    \Sc F_{\partial\Omega}[\vec \alpha-\Hi[\alpha_0]] &=
    -\Ti[\vec g]=-\Fi[h_0].   
\end{align*}

Note that \eqref{eq:formula_BP} applied to the function 
$T_\Omega[\vec g]$, and the fact that $DT_\Omega[\vec g]=\vec g$ yield
$F_{\partial\Omega}[\alpha_0+\vec\alpha]=0$, so
\begin{align}\label{eq:nt-limit}
   \tr{ F_{\partial\Omega}[\vec \alpha-\Hi[\alpha_0]]}
   =-\tr F_{\partial\Omega}[\alpha_0+\Hi[\alpha_0]] 
  =-\alpha_0-\Hi[\alpha_0],
\end{align}
which proves the second equality in the solution
\eqref{eq:sol-div-curl}.  The fact that $\vec w$ belongs to the
Sobolev space $W^{1,p}(\Omega,\R^3)$ is a direct consequence of
Theorem \ref{theo:T_acotado} and \eqref{eq:Cauchy-operator}.
\end{proof}
  
The following result gives us an alternative way to complete a
scalar-valued harmonic function to a monogenic function, similarly to
the way the radial integration operator
$\overrightarrow{U}_{\!\!\Omega}$ in \cite[Proposition
2.3]{DelPor2017} did this for star-shaped domains. It can be
considered an ``interior'' version of the construction of the Hilbert
transform $\Hi$, in other words, a method to construct harmonic
conjugates in Lipschitz domains of $\R^3$. See also the classical
generalization of harmonic conjugates using SI-vector fields in the
upper half space of $\R^n$ \cite{SW1960}. In this sense we can state
the follows

\begin{corolario}
  Let $\Omega$ be as in Theorem \ref{theo:div-curl}. Let
  $w_0\in W^{1,p}(\Omega,\R)$ be a scalar harmonic function. Let
\begin{align*}
   \vec w=\Fii[2(I+K_0)^{-1}\tr w_0].
\end{align*} 
Then $w_0+\vec w$ is monogenic in $\Omega$.

\end{corolario}

\begin{corolario}
 Let $\Omega$ be as in Theorem \ref{theo:div-curl}. The following
  is a right inverse of $\curl$:
\begin{align}\label{eq:inverse-curl}
   \vec g \mapsto \Tiii[\vec g] - \Fii[2(I+K_0)^{-1}\alpha_0],
\end{align}
acting on all $\vec g$ in the class of divergence free vector fields.
\end{corolario}

Since the right inverse of curl \eqref{eq:inverse-curl} acts as
\[   \Tiii - 2 \Fii(I+K_0)^{-1}\tr \Ti \colon \Sol(\Omega,\R^3) 
         \to \Sol(\Omega,\R^3),
\]
we have
\begin{corolario}
  Let $\Omega$ be as in Theorem \ref{theo:div-curl}. The following is
  a right inverse for the double curl operator:
\begin{align*}
  \vec g \mapsto  -L[\vec g] + M[2(I+K_0)^{-1} \alpha_0 \eta],
\end{align*}
for every $\vec g$ in the class of divergence free vector fields,
where $\eta$ the outward pointing unit normal vector to
$\partial\Omega$.
\end{corolario}

\begin{proof}
  This is a direct consequence of \eqref{eq:F0_F1-div-curl} and 
  the fact that $\Tiii[\vec g]=-\curl L[\vec g]$.
\end{proof}

\newpage

\end{document}